\documentclass[11pt,reqno]{amsart}

\usepackage{times}

\usepackage[
  margin=1in
]{geometry}

\usepackage{amssymb,amsfonts,amsthm}
\usepackage{comment}
\usepackage{enumitem}
\usepackage{xfrac}

\usepackage[dvipsnames]{color}

\usepackage[colorlinks=true, pdfstartview=FitV, linkcolor=black, citecolor=blue, urlcolor=blue]{hyperref}

\def\Xint#1{\mathchoice
{\XXint\displaystyle\textstyle{#1}}%
{\XXint\textstyle\scriptstyle{#1}}%
{\XXint\scriptstyle\scriptscriptstyle{#1}}%
{\XXint\scriptscriptstyle%
\scriptscriptstyle{#1}}%
\!\int}
\def\XXint#1#2#3{{\setbox0=\hbox{$#1{#2#3}{%
\int}$ }
\vcenter{\hbox{$#2#3$ }}\kern-.6\wd0}}
\def\barint{\, \Xint -} 
\def\bariint{\barint_{} \kern-.4em \barint}
\def\bariiint{\bariint_{} \kern-.4em \barint}
\renewcommand{\iint}{\int_{}\kern-.34em \int} 
\renewcommand{\iiint}{\iint_{}\kern-.34em \int} 

\DeclareMathAlphabet{\mathcal}{OMS}{cmsy}{m}{n}

\theoremstyle{plain}

\newtheorem{theorem}{Theorem}[section]

\newtheorem{definition}[theorem]{Definition}
\newtheorem{lemma}[theorem]{Lemma}

\newtheorem{corollary}[theorem]{Corollary}
\newtheorem{proposition}[theorem]{Proposition}

\theoremstyle{definition}
\newtheorem{remark}[theorem]{Remark}

\newcommand{\R}{\mathbb{R}}

\newcommand{\N}{\mathbb{N}}
\newcommand{\Z}{\mathbb{Z}}
\newcommand{\T}{\mathbb{T}}

\newcommand{\X}{\mathbf{X}}
\newcommand{\Y}{\mathbf{Y}}
\newcommand{\bZ}{\mathbf{Z}}

\newcommand{\p}{\partial}
\newcommand{\la}{\langle}

\newcommand{\ra}{\rangle}
\newcommand{\les}{\lesssim}
\newcommand{\ges}{\gtrsim}
\newcommand{\norm}[1]{\left\| #1 \right\|}
\renewcommand{\:}{\colon}

\let\div\relax
\DeclareMathOperator{\div}{div}

\DeclareMathOperator{\sgn}{sgn}
\let\tilde\relas
\newcommand{\tilde}[1]{\widetilde{#1}}

\newcommand{\red}[1]{{\color{red} #1}}

\usepackage[title]{appendix}

\numberwithin{equation}{section}
\newcommand{\nuhalf}{\nu^{\sfrac{1}{2}}}
\newcommand{\nuhalfk}{\nu^{\sfrac{k}{2}}}

\newcommand{\ad}{\textnormal{ad}}

\title{Enhanced dissipation and H{\"o}rmander's hypoellipticity}

\author{Dallas Albritton}
\address{Courant Institute of Mathematical Sciences, New York University, New York, NY 10012}
\email{daa399@cims.nyu.edu}
\author{Rajendra Beekie}
\address{Courant Institute of Mathematical Sciences, New York University, New York, NY 10012}
\email{beekie@cims.nyu.edu}
\author{Matthew Novack}
\address{Courant Institute of Mathematical Sciences, New York University, New York, NY 10012}
\email{mdn7@cims.nyu.edu}

\begin{document}

\maketitle

\begin{abstract}
We examine the phenomenon of enhanced dissipation from the perspective of H{\"o}rmander's classical theory of second order hypoelliptic operators~\cite{hormander67}. Consider a passive scalar in a shear flow, whose evolution is described by the advection--diffusion equation
\[ \p_t f + b(y) \p_x f - \nu \Delta f = 0 \text{ on } \T \times (0,1) \times \R_+ \]
with periodic, Dirichlet, or Neumann conditions in $y$. We demonstrate that decay is enhanced on the timescale $T \sim \nu^{-(N+1)/(N+3)}$, where $N-1$ is the maximal order of vanishing of the derivative $b'(y)$ of the shear profile and $N=0$ for monotone shear flows. In the periodic setting, we recover the known timescale of Bedrossian and Coti Zelati~\cite{jacobmicheleshear}. Our results are new in the presence of boundaries.
\end{abstract}


\setcounter{tocdepth}{1}
\tableofcontents

\section{Introduction}

The advection--diffusion equation
\begin{equation}
\label{eq:fpde1}
\begin{aligned}
\p_t f + \boldsymbol{u} \cdot \nabla f -  \nu \Delta f &= 0 \\
\end{aligned}
\end{equation}
describes the evolution of a passive scalar $f$ which is transported by a divergence-free velocity field $\boldsymbol{u}$ and simultaneously diffuses. It is well known that the interplay between transport $\boldsymbol{u} \cdot \nabla$ and weak diffusion $\nu \Delta$ ($0 < \nu \ll 1$) may cause solutions of~\eqref{eq:fpde1} to decay more rapidly than 
by diffusion alone. This phenomenon is known as \emph{enhanced dissipation}. 

Our goal is to clarify the relationship between enhanced dissipation and H{\"o}rmander's classical theory of hypoelliptic second order operators, developed in~\cite{hormander67}. We restrict our attention to the class of \emph{shear flows} $\boldsymbol{u} = (b(y),0)$. 
Enhanced dissipation within this class has already been studied by Bedrossian and Coti Zelati in~\cite{jacobmicheleshear}. We will sharpen and extend their results, though perhaps the main interest of our work is in the method we employ, which is a new application of the old ideas of H{\"o}rmander. 

We consider~\eqref{eq:fpde1} on the spacetime domain $\Omega \times (0,+\infty)$, where $\Omega = \T \times D$ and $D = \T$ or $(0,1)$. By convention, we identify functions on the torus $\T:=\R/\Z$ with functions  satisfying periodic boundary conditions on $(0,1)$. When $D = (0,1)$, we impose homogeneous Dirichlet or Neumann conditions on $\{ y = 0,1 \}$. The shear profile $b \in C^\infty(\bar{D})$ is assumed to have precisely $M \geq 0$ critical points. When $M \geq 1$, we denote these by $y_i$, $1 \leq i \leq M$. We further require that the maximal order of vanishing $N_i \in \N$ at $y_i$ is finite:
\begin{subequations}\label{eq:crit:points}
\begin{align}
    b^{(k)}(y_i) = 0 \text{ for all } 1 \leq k \leq N_i \label{eq:crit:points:one} \\ 
    b^{(N_i+1)}(y_i) \neq 0 \, . \label{eq:crit:points:two}
\end{align}
\end{subequations}
Let $N = \max_{1 \leq i \leq M} N_i$, or $N = 0$ when $M = 0$. 

Since $\boldsymbol{u} \cdot \nabla = b(y) \p_x$ in~\eqref{eq:fpde1}, the $x$-averages
\begin{equation}\notag
    \la f \ra(y,t) = \int_{\T} f(x,y,t) \, dx
\end{equation}
satisfy the one-dimensional heat equation
\begin{equation}\notag
    \p_t \la f \ra - \nu \p_{yy} \la f \ra = 0 \, 
\end{equation}
and in general decay only on the diffusive timescale $\sim \nu^{-1}$. To extract an enhanced dissipation timescale, it is necessary to subtract $\langle f \rangle$. Without loss of generality, we thus require that the initial data $f_{\rm in}$ satisfies
\begin{equation}\notag
    \la f_{\rm in} \ra(y) := \int_{\T} f_{\rm in}(x,y) \, dx \equiv 0 \, .
\end{equation}

It is instructive to extract the dissipation timescale in the Couette example $b(y) = y$ posed for $(x,y)\in(\R / 2\pi \Z) \times \R$. In this case,~\eqref{eq:fpde1} is explicitly solvable, as was already known to Kelvin in~\cite{kelvin1887stability}. Upon taking the Fourier transform, we have
\begin{equation}\notag
    \p_t \widehat{f} - k \p_\eta \widehat{f} + \nu (|k|^2 + |\eta|^2) \widehat{f} = 0 \, .
\end{equation}
The solution is
\begin{equation}
    \label{eq:couetteexactsol}
    \widehat{f}(k,\eta,t) = e^{-\nu k^2 t} e^{-\nu t (\eta^2 + \eta k t + k^2 t^2/3)} \widehat{f_{\rm in}}(k,\eta + kt) \, .
\end{equation}
The factor $e^{-\nu |k|^2 t^3/3}$ dominates for large times and yields the enhancement timescale $T \sim \nu^{-1/3}$, since $|k| \ges 1$ when we disregard the $k = 0$ modes. The fundamental solution on $\R^2 \times \R_+$ was also famously computed by Kolmogorov~\cite{kolmogorov} (see also H{\"o}rmander's introduction~\cite{hormander67}).


In~\cite{jacobmicheleshear}, Bedrossian and Coti Zelati identified the following enhanced dissipation timescales depending on the shear profile:
\begin{equation}
    \label{eq:thetimescale}
    T \sim \nu^{-\frac{N+1}{N+3}}
\end{equation}
in the periodic-in-$y$ setting, up to a factor $(1 + \log \nu^{-1})^2$ which was removed by Wei~\cite{dongyiweiresolvent}. The enhancement is measured in terms of the `flatness' of the shear profile at its critical points. Since the Couette flow has no critical points, the above timescale for $N=0$ matches the timescale extracted from \eqref{eq:couetteexactsol}. These timescales are essentially sharp, as demonstrated by Drivas and Coti Zelati in~\cite{drivascotizelati}. In the Neumann setting, Bedrossian and Coti Zelati obtain the same timescale except in the monotone case $N=0$, in which the timescale $T \sim \nu^{-1/2}$ (up to a logarithmic correction) is obtained instead of $T \sim \nu^{-1/3}$. We demonstrate that the enhanced dissipation timescale is, at worst,~\eqref{eq:thetimescale} in the Dirichlet and Neumann settings for all $N \in \N_0$.


To better clarify the connections to hypoellipticity, we consider the \emph{hypoelliptic} advection-diffusion equation proposed in~\cite{jacobmicheleshear}:
\begin{equation}
\label{eq:fpde}
\begin{aligned}
\p_t f + b(y) \p_x f -  \nu \p_{yy} f &= 0.
\end{aligned}
\end{equation}
The PDEs~\eqref{eq:fpde1} and~\eqref{eq:fpde} are actually \emph{equivalent} in the following sense: If $f$ is a solution of~\eqref{eq:fpde} with initial data $f_{\rm in}$, then $g = e^{\nu t \p_{xx}} f$ is a solution of~\eqref{eq:fpde1}. In the sequel, we focus only on smooth solutions of~\eqref{eq:fpde}: For each $f_{\rm in} \in C^\infty(\bar{\Omega})$ satisfying the periodic, Dirichlet, or Neumann boundary conditions, there exists a unique smooth solution $f \in C([0,+\infty);L^2(\Omega)) \cap  C^\infty(\bar{\Omega} \times (0,+\infty))$ to~\eqref{eq:fpde} satisfying $f(\cdot,0) = f_{\rm in}$ and the desired boundary conditions. Smoothness up to $t= 0$ would require additional compatibility conditions on $f_{\rm in}$. \emph{We assume that $\la f \ra = \la f_{\rm in} \ra \equiv 0$.}


Our main theorems are the following:

\begin{theorem}[Enhanced dissipation]\label{thm:enhancement}
There exist $d_0, \nu_0 > 0$ and $C>1$, depending only on $b$, such that the solution $f$ of~\eqref{eq:fpde} with periodic, Dirichlet, or Neumann conditions satisfies
\begin{equation}\notag
	\norm{f(\cdot,t)}_{L^2(\Omega)} \leq C \norm{f_{\rm in}}_{L^2(\Omega)} \exp\left( - d_0 \nu^{\frac{N+1}{N+3}} t \right)
\end{equation}
for all $t > 0$ and $\nu \in (0,\nu_0)$.
\end{theorem}

\begin{theorem}[Gevrey-in-$x$ smoothing]\label{thm:smoothing}
There exist $d_0, \nu_0 > 0$ depending only on $b$ such that the solution $f$ of~\eqref{eq:fpde} is instantly Gevrey-$p$ regular for all $p>\frac{N+3}{2}$ and $\nu \in (0,\nu_0)$:
\begin{equation}\notag
	\left\lVert \exp \left( d_0 \nu^{\frac{N+1}{N+3}} |\p_x|^{\frac{1}{p}} t \right) f(\cdot,t) \right\rVert_{L^2(\Omega)} \les_{b,p} \norm{f_{\rm in}}_{L^2(\Omega)}.
\end{equation}
\end{theorem}

\begin{remark}[Improvements]
Theorems~\ref{thm:enhancement} and~\ref{thm:smoothing} improve the main theorems of Bedrossian and Coti Zelati~\cite{jacobmicheleshear} and Wei~\cite{dongyiweiresolvent} in the following ways:
\begin{itemize}[leftmargin=*]
    \item We treat Dirichlet conditions in addition to periodic and Neumann conditions.
    \item In the Neumann setting, we improve the timescale from $T \sim \nu^{-1/2}$ in~\cite{jacobmicheleshear} to $T \sim \nu^{-1/3}$ when $N=0$, and we remove the additional logarithmic factors when $N \in \N_0$.
\end{itemize}
\end{remark}

In the periodic-in-$y$ setting, $p=\frac{N+3}{2}$ is also possible in Theorem~\ref{thm:smoothing}, as demonstrated by Wei~\cite{dongyiweiresolvent} using resolvent estimates and a  Gearhardt--Pr{\"u}ss-type theorem.

\begin{remark}[Generalizations]
Our approach generalizes, with minimal effort, to divergence-form operators $\p_y (a(y) \p_y \cdot )$ satisfying $a \in L^\infty(D)$ and $a \geq c > 0$. It is also possible to generalize to operators $\div (A(y) \nabla \cdot)$. The norms in the `interpolation inequality',  Proposition~\ref{lemma:Z0:estimate}, would need to be adjusted to incorporate $x$ derivatives.\footnote{Roughly speaking, one should measure $\| \nabla f \|_{L^2}$ and $\| X_0 f \|_{L^2_t \dot H^{-1}_{x,y}}$ instead of $\| \p_y f \|_{L^2}$ and $\| X_0 f \|_{L^2_{t,x} \dot H^{-1}_y}$.} Doing so would allow us to treat circular flows $\boldsymbol{u} = r \Omega(r) e_\theta$ in an annulus $r_0 < r < r_1$, where $\Omega$ has finitely many critical points. It is also possible to generalize to shear flows which vary on the enhancement timescale $\nu^{-\frac{N+1}{N+3}}$, though the time translation symmetry is convenient in the proof of Proposition~\ref{lemma:Z0:estimate}. We view these generalizations  as an indication that, in a certain sense, H{\"o}rmander's method is robust.
\end{remark}

The most na{\"i}ve approach to Theorem~\ref{thm:enhancement} is to try to generalize the $L^2$ energy estimate for the heat equation $\p_t f = \nu \Delta f$ on $\T^2 \times \R_+$:
\begin{equation}
    \label{eq:themostbasicenergyestimate}
    \frac{1}{2} \frac{d}{dt} \| f \|_{L^2}^2 = - \nu \| \nabla f \|_{L^2}^2 \leq - C_P^{-1} \nu \| f \|_{L^2}^2 \, ,
\end{equation}
where $C_P = 4\pi^2$ is the Poincar{\'e} constant and $f$ has zero mean. This differential inequality implies $\| f \|_{L^2}^2 \leq e^{-\nu t/C_P} \| f_{\rm in} \|_{L^2}^2$. For the hypoelliptic PDE~\eqref{eq:fpde}, $- \nu \| \p_y f \|_{L^2}^2$ appears on the right-hand side of the energy estimate, and we lack a Poincar{\'e} inequality for this quantity. The goal of hypoellipticity is precisely to extract differentiability-in-$x$ and, thus, a Poincar{\'e} inequality, from the $L^2$ energy estimate and the PDE. In this way, we salvage the na{\"i}ve approach.

The desire to reproduce~\eqref{eq:themostbasicenergyestimate} is also a motivation for hypocoercivity. We compare the two approaches after we sketch the main idea. \\ 

In the classical terminology, a partial differential operator $P$ with smooth coefficients on an open set $U \subset \R^d$ is \emph{hypoelliptic} if distributional solutions of $Pu = f$ are smooth in $U$ for $f \in C^\infty(U)$. 
 Let $X_0, X_1, \hdots, X_J$ be smooth vector fields on $U$.
 In~\cite{hormander67}, H{\"o}rmander characterized hypoellipticity for second order operators
\begin{equation}\notag
    L = \sum_{j=1}^J X_j^* X_j + X_0 \, .
\end{equation}
$L$ is hypoelliptic \emph{if and only if} the Lie algebra generated by $\{ X_j \}_{j=0}^J$ spans $\R^d$ at every point $x_0 \in U$. 
Perhaps the main observation in~\cite{hormander67} is that  regularity `along' two vector fields $X$ and $Y$ implies regularity along their commutator $[X,Y]$ and sum $X+Y$, up to certain controllable errors.

H{\"o}rmander's theorem implies that solutions to~\eqref{eq:fpde} on $\T^2 \times \R_+$ are smooth, but it does not evidently imply Theorem~\ref{thm:enhancement}. Rather, we revisit H{\"o}rmander's scheme (i) in a more quantitative manner, (ii) with boundary, and (iii) in the singular limit as $\nu \to 0^+$. On the other hand, our problem admits technical simplifications that, we hope, illuminate the original analysis in~\cite{hormander67}.

Our work is partially inspired by the recent work of Armstrong and Mourrat~\cite{am19}, who applied H{\"o}rmander's methods in the context of the kinetic Fokker--Planck equation. H{\"o}rmander's analysis was also also recently revisited by Bedrossian and coauthors in~\cite{bedrossian2020regularity,bedrossian2020quantitative} in the context of stochastic differential equations.


\subsection*{Sketch}


Let us provide an exposition of the main steps in our proof. Consider a smooth solution $f \: \Omega \times \R_+ \to \R$ to~\eqref{eq:fpde} with periodic or homogeneous Dirichlet or Neumann boundary conditions and smooth initial data $f_{\rm in}$ satisfying $\la f_{\rm in} \ra \equiv 0$. Using the boundary conditions, the basic energy estimate is
\begin{equation}\notag
    \frac{1}{2} \| f(\cdot,t) \|_{L^2}^2 + \nu \int_0^t \int_{\Omega} |\p_y f|^2 \, dx \, dy \, ds = \frac{1}{2} \| f_{\rm in} \|_{L^2}^2, \quad \forall t \geq 0 \, .
\end{equation}
This energy estimate furnishes two crucial pieces of information, which, to avoid more notation, we present here only for Dirichlet boundary conditions. Evidently,
\begin{equation}
    \label{eq:crucialest1}
    \nu \| f \|_{L^2_{t,x} (H^1_0)_y(\Omega \times \R_+)}^2 \les \| f_{\rm in} \|_{L^2}^2 \, .
\end{equation}
Then~\eqref{eq:crucialest1} and the PDE itself $X_0 f := (\p_t + b(y) \p_x) f = \nu \p_{yy} f$ give
\begin{equation}
    \label{eq:crucialest2}
    \nu^{-1} \| X_0 f \|_{L^2_{t,x} H^{-1}_y(\Omega \times \R_+)}^2 \les \| f_{\rm in} \|_{L^2}^2 \, .
\end{equation}
 As in~\cite{am19}, we could define a norm $\| \cdot \|_{H^{1,\nu}_{\rm hyp}}$ based on the quantities controlled in~\eqref{eq:crucialest1} and \eqref{eq:crucialest2}, though it is not strictly necessary.\footnote{Membership in $H^{1,\nu}_{\rm hyp}$ is actually equivalent to solving the PDE with RHS in $L^2_{t,x} H^{-1}_y$.} Hypoellipticity is subsequently extracted from \emph{embedding theorems} for $H^{1,\nu}_{\rm hyp}$.\footnote{As a consequence, after the initial energy estimate, time no longer plays a distinguished role. It is possible to distinguish the role of time in a `parabolic' H{\"o}rmander method; see, for example, Bedrossian and Liss~\cite{bedrossian2020quantitative}.}  Notably, we have
 \begin{proposition}[Subelliptic estimate and Poincar{\'e} inequality, Dirichlet case]
 \label{pro:introsubelliptic}
\begin{equation}
    \label{eq:introsubellipticestimate}
        \nu^\frac{N+1}{N+3} \| f \|_{L^2(\Omega_T)}^2 + \nu^\frac{N+1}{N+3}  \| f   \|_{Q^{1/(N+3)}_{\p_x}(\Omega_T)}^2 \les_b \nu \| f \|_{L^2_{t,x} H^1_{0,y}(\Omega_T)}^2 + \nu^{-1} \| X_0 f \|_{L^2_{t,x} H^{-1}_y(\Omega_T)}^2 \, ,
\end{equation}
where $\Omega_T = \Omega \times (0,T)$, $T = T_0 \nu^{-\frac{N+1}{N+3}}$, $\nu \leq \nu_0$, and $\nu_0, T_0 > 0$ depend only on the shear profile $b$.
\end{proposition}
The quantity $\| \cdot \|_{Q^{1/(N+3)}_{\p_x}}$ is a Besov-type norm measuring $1/(N+3)$-differentiability in the direction $\p_x$, and it controls $\| f \|_{L^2}$ when $\la f \ra \equiv 0$ as above. In the terminology of~\cite{am19},~\eqref{eq:introsubellipticestimate} is a \emph{H{\"o}rmander inequality}. See Theorem~\ref{thm:improvedsubellipticestimate} for the periodic, Dirichlet, and Neumann settings. \\
 
 With~\eqref{eq:introsubellipticestimate} in hand, we conclude the proof of exponential decay according to a Gr{\"o}nwall-type argument along a discrete set of times (see Section~\ref{sec:proofofmaintheorem}). The discrete-in-time proof is natural in the following sense. For the Couette example~\eqref{eq:couetteexactsol}, the differential inequality $\sfrac{d}{dt} \| f(\cdot,t) \|_{L^2}^2 \leq - C_0 \nu^{\sfrac{1}{3}} \| f(\cdot,t) \|_{L^2}^2$ would imply $\| f(\cdot,t) \|_{L^2}^2 \leq e^{-C_0 \nu^{\sfrac{1}{3}} t} \| f_{\rm in} \|_{L^2}^2$, which is false, as we see from the solution formula~\eqref{eq:couetteexactsol}. Indeed, the superlinear decay in the term $e^{-\nu k^2 t^3/3}$ indicates that `nothing happens' until the `crossover time', which is proportional to $\nu^{-\sfrac{1}{3}}$. In light of this, the prefactor $C> 1$ in Theorem~\ref{thm:enhancement} and the scaling of $T_0$ in Proposition~\ref{pro:introsubelliptic} are \emph{essential}.

\subsection*{Sample H{\"o}rmander inequalities}

To convey the main idea of the embedding~\eqref{eq:introsubellipticestimate}, we demonstrate a simple embedding on $\R^2$, which is an analogue of~\eqref{eq:introsubellipticestimate} for the PDE $y\p_x f - \nu \p_{yy} f = g$. To quantify the fractional differentiability, we define the following Besov-type seminorms based on finite differences:
\begin{equation}\notag
    \| f \|_{Q^1_{\p_y}} := \sup_{h\in \R \setminus \{0\}} h^{-1} \| f(x,y+h) - f(x,y) \|_{L^2}
\end{equation}
\begin{equation}\notag
    \| f \|_{Q^{\sfrac{1}{2}}_{y \p_x}} := \sup_{h \in \R \setminus \{0\}} h^{-1} \| f(x+ \sgn(h) h^2 y,y) - f(x,y) \|_{L^2}
\end{equation}
\begin{equation}\notag
    \| f \|_{Q^{\sfrac{1}{3}}_{\p_x}} := \sup_{h \in \R \setminus \{0\}} h^{-1} \| f(x+h^3,y) - f(x,y) \|_{L^2} \, .
\end{equation}
We evidently have $\| f \|_{Q^1_{\p_y}} \les \|  f \|_{L^2_x \dot H^1_y}$ (see the characterization of Sobolev norms by finite differences in Evans \cite{evans}). The quantity $\| f \|_{Q^{\sfrac{1}{2}}_{y \p_x}}$, which measures $\sfrac{1}{2}$ derivative in the $y\p_x$ direction, will be controlled by the `interpolation inequality' below. The analogous estimate for the heat equation is  $f \in L^2_t H^1_x \cap H^1_t H^{-1}_x \subset H^{\sfrac{1}{2}}_t L^2_x$ on $\R^d \times (0,T)$.

In the following, we require that the terms on the right-hand side of the inequalities~\eqref{eq:subellipticestimateintro},~\eqref{eq:bracketinequalityintro} and~\eqref{eq:interpolationinequalityintro} are \emph{well defined and non-vanishing} (which is necessary in the optimization steps). Ultimately, we will demonstrate

\begin{lemma}[Sample subelliptic estimate]\label{lem:sample:sub}
In the above notation,
    \begin{equation}
        \label{eq:subellipticestimateintro}
        \| f \|_{Q^{\sfrac{1}{3}}_{\p_x}} \les \| f \|_{L^2_x \dot H^1_y}^{\sfrac{2}{3}} \| y\p_x f \|_{L^2_x \dot H^{-1}_y}^{\sfrac{1}{3}} \ .
    \end{equation}
\end{lemma}

The explains, in part, the emergence of the powers of $\nu$ in the main theorem. If we suppose $\| f \|_{L^2_x \dot H^1_y} \les \nu^{-\sfrac{1}{2}}$ and $\| f \|_{L^2_x \dot H^{-1}_y} \les \nu^{\sfrac{1}{2}}$, then we recover $\| f \|_{Q^{\sfrac{1}{3}}_{\p_x}} \les \nu^{-\sfrac{1}{6}}$, analogous to the time-dependent subelliptic estimate~\eqref{eq:introsubellipticestimate} with $N=0$. If we replace $y\p_x f$ by $y^{N+1} \p_x f$ in~\eqref{eq:subellipticestimateintro}, then we would expect
\begin{equation}
        \| f \|_{Q^{1/(N+1)}_{\p_x}} \les \| f \|_{L^2_x \dot H^1_y}^{\frac{N+2}{N+3}} \| y^{N+1} \p_x f \|_{L^2_x \dot H^{-1}_y}^{\frac{1}{N+3}} \, ,
    \end{equation}
though we do not demonstrate this here (see Farkas and Lunardi~\cite{lunardi2006} for similar estimates).

The subelliptic estimate~\eqref{eq:subellipticestimateintro} follows from combining the `bracket inequality' and the `interpolation inequality'. The bracket inequality is based on the identity $e^{-Bt}e^{-At}e^{Bt}e^{At} = e^{t^2 [A,B] + O(t^3)}$ for vector fields $A$, $B$. In the proof, this inequality manifests in~\eqref{eq:basicbracketcomputation}, where the $O(t^3)$ term vanishes due to the special structure of $\p_y$ and $y \p_x$.


\begin{lemma}[Sample bracket inequality]
In the above notation,
    \begin{equation}
        \label{eq:bracketinequalityintro}
    \| f \|_{Q^{\sfrac{1}{3}}_{\p_x}} \les \| f \|_{Q^1_{\p_y}}^{\sfrac{1}{3}} \| f \|_{Q^{\sfrac{1}{2}}_{y \p_x}}^{\sfrac{2}{3}} \, .
    \end{equation}
\end{lemma}
\begin{proof}
Let $h \in \R$. If we demonstrate
\begin{equation}
    \label{eq:intermediarystepandwhatnot}
    \| f \|_{Q^{\sfrac{1}{3}}_{\p_x}} \les \| f \|_{Q^1_{\p_y}} + \| f \|_{Q^{\sfrac{1}{2}}_{y \p_x}} \, ,
\end{equation}
then we consider $f_\lambda(x,y) = f(\lambda x,y)$ and optimize $\lambda > 0$ to complete the proof.
We write
\begin{equation}
    \label{eq:basicbracketcomputation}
\begin{aligned}
 f(x+h^3,y) - f(x,y) &= f(x+h^3,y) - f(x+(y+h)h^2,y) \\
 &+ f(x+(y+h)h^2,y) - f(x+(y+h)h^2,y+h) \\
 &+ f(x+(y+h)h^2,y+h) - f(x,y+h) \\
 &+ f(x,y+h) - f(x,y) \, .
\end{aligned}
\end{equation}
First, we treat the terms based on `flowing' in the direction $\mp \p_y$:
\begin{equation}\notag
    \| f(x+(y+h)h^2,y) - f(x+(y+h)h^2,y+h) \|_{L^2} = \| f(x,y+h) - f(x,y) \| \les h \| f \|_{Q^1_{\p_y}} \, .
\end{equation}
Second, we treat the terms based on `flowing' in the direction $\mp y \p_x$:
\begin{equation}\notag
    \| f(x+h^3,y) - f(x+(y+h)h^2,y) \|_{L^2} = \| f(x+(y+h)h^2,y+h) - f(x,y+h) \|_{L^2} \les h \| f \|_{Q^{\sfrac{1}{2}}} \, .
\end{equation}
Combining the estimates completes the proof of~\eqref{eq:intermediarystepandwhatnot}.
\end{proof}

The interpolation inequality is more difficult (see Section~5 in H{\"o}rmander's paper~\cite{hormander67}), and we delay the proof to Appendix~\ref{app:interpolationproof}:

\begin{lemma}[Example interpolation inequality]
In the above notation,
    \begin{equation}
    \label{eq:interpolationinequalityintro}
    \| f \|_{Q^{\sfrac{1}{2}}_{y \p_x}} \les \| f \|_{L^2_x \dot H^1_y}^{\sfrac{1}{2}} \| y \p_x f \|_{L^2_x \dot H^{-1}_y}^{\sfrac{1}{2}} \, .
    \end{equation}
\end{lemma}

The main difficulty in~\eqref{eq:interpolationinequalityintro} can be seen in the following way. Let $X_0 = y\p_x$ and $e^{h^2 X_0} f = f(x+h^2y,y)$. To achieve $\| e^{h^2 X_0} f - f \|_2^2 \les h^2$, one differentiates the left-hand side in $h$. Then
\begin{equation*}
\begin{aligned}
	\frac{d}{dh} \| e^{h^2 X_0} f - f \|_2^2  = 4 h \langle e^{h^2 X_0} f - f, X_0 f \rangle \, ,
	\end{aligned}
\end{equation*}
which is estimated by duality: $X_0 f \in L^2_x \dot H^{-1}_y$ and $e^{h^2 X_0} f - f \in L^2_x \dot H^1_y$. However, $\| \p_y f \|_{L^2}$ is not generally preserved under the flow of $X_0$, as we must commute $\p_y$ with $e^{h^2 X_0}$. The chain rule gives $\p_y e^{h^2 X_0} f = h^2 \p_x e^{h^2 X_0} f + e^{h^2 X_0} \p_y f$, and the situation looks bleak because $\p_x f$ is uncontrolled in $L^2$. H{\"o}rmander's method involves the above computation for a smooth function $S_\ell f$, obtained by mollifying $f$ in the direction $[Y,X_0] = \p_x$ at scale $\ell > 0$. The additional mollification  makes it possible to exploit the small prefactor $h^2$. The complete proof is contained in Appendix~\ref{app:interpolationproof}.

\subsection*{Further difficulties}
Foremost is the Taylor remainder term, which vanished in the calculation~\eqref{eq:basicbracketcomputation}. In our setting, the left-hand side of the analogous bracket calculation contains $f(x + (b(y+h) - b(y))h^2,y)$ rather than $f(x + b'(y) h^3, y)$. This error is estimated in terms of further commutators. That is, to control $\sfrac{1}{3}$ derivative in the direction $b'(y) \p_x$, we already require control of $\sfrac{1}{4}$ derivative in the direction $b''(y)\p_x$, $\sfrac{1}{5}$ derivative in the direction $b^{(3)} \p_x$, etc. Eventually, these errors can be absorbed by applying the trick $[Y, X_0] = [\varepsilon^{-1} Y, \varepsilon X_0 ]$ to introduce small constants. 

The interpolation inequality also exhibits a `cascading effect' and requires mollification along the vector fields $[Y, X_0]$, $[Y, [Y, X_0]]$, etc. Finally, the boundaries in $y$ when $D=(0,1)$ and in the time $t$ restrict the `admissible finite differences' $h \in \R \setminus \{ 0 \}$. One requires $|h| \leq h_0$ with judicious $\nu$-dependent choices of $h_0$. We emphasize that the more tedious calculations present in this paper essentially result from either Taylor remainders, interpolation, or admissible finite differences.  The scaling with respect to $\nu$ which produces the enhanced dissipation timescale is already visible in the discussion following Lemma~\ref{lem:sample:sub} and is unrelated to these technicalities.

\subsection*{Comparison with hypocoercivity}
A commonly used method to prove enhanced dissipation~\cite{jacobmicheleshear} is \emph{hypocoercivity}, whose mathematical framework was developed by Villani in~\cite{villanihypocoercivity}. Hypocoercivity has roots in kinetic theory\footnote{See also the weighted energy method of Guo in~\cite{guolandau}.} and was later adapted to PDEs of fluid dynamics in~\cite{gallaghergallaynier,beckwaynebar}. Its premise is to design a new energy functional $\Phi$ which satisfies the differential inequality $\dot \Phi \leq - C_0 \Phi$. Commonly, $\Phi$ is equivalent to $\| \cdot \|_{H^1_{x,y}}^2$ and is `augmented' by certain commutators; for example,
\begin{equation}\notag
    \Phi[f] := \| f \|_2^2 + a \nu^{2/3} \| \p_y f \|_2^2  + 2b\nu^{1/3} \langle \p_x f, \p_y f \rangle + c\| \p_x f \|_2^2
\end{equation}
in the Couette case with suitable weights $0 < c \ll b \ll a \ll 1$. We should interpret $\p_x$ as the commutator $[\p_y, y \p_x]$ above. Perhaps the main advantage of hypocoercivity is that it is elementary, consisting only of well-chosen energy estimates and integration by parts. 

Hypocoercivity naturally lives at higher regularity $\approx L^\infty_t H^1_{x,y}$ than H{\"o}rmander's method, which lives at the level of the basic energy estimate $\approx L^\infty_t L^2_{x,y} \cap L^2_{t,x} H^1_y$. As a consequence, hypocoercive methods encounter two technical difficulties that are absent from our approach:
\begin{itemize}[leftmargin=*]
    \item $L^2$ stability estimates require further arguments that exploit the smoothing. In~\cite{gallaghergallaynier,jacobmicheleshear}, this is addressed by passing through the resolvent formalism: The hypocoercivity method estimates the pseudospectrum and, therefore, the semigroup.  It is this step which contributes the logarithmic correction, which can be removed by the Gearhardt--Pr{\"u}ss-type theorem of Wei in~\cite{dongyiweiresolvent}.\footnote{Notice, however, that resolvent-based methods may not generalize easily to unsteady flows. Our approach can be easily adapted to shear flows that vary on the timescale $O(\nu^{-(N+1)/(N+3)})$.} In the monotone shear case, one can use the smoothing more directly~\cite{stablemixingestimates}. A different choice, popular in nonlinear problems, is simply to consider stability in more regular spaces~\cite{beckwaynebar}.
    \item Hypocoercivity may not always generalize well to boundaries, which is why~\cite{jacobmicheleshear} does not achieve the optimal timescales in the Neumann case.
\end{itemize}

Finally, we mention that hypocoercivity's main purpose is not to capture hypoelliptic regularization, although in the shear flow case it can be recovered by studying the equation mode-by-mode in $x$; see (2.6) in~\cite{jacobmicheleshear} for a definition of the corresponding functional $\Phi_k$. In our method, the enhanced dissipation is essentially a \emph{consequence} of the hypoelliptic regularization, in the sense that the subelliptic estimate \eqref{eq:introsubellipticestimate} controls a norm measuring fractional regularity in $x$ using the energy estimates \eqref{eq:crucialest1} and \eqref{eq:crucialest2} and leads directly to the timescale \eqref{eq:thetimescale} by a Poincar{\'e} inequality in $x$.

\subsection*{Existing literature} \emph{1. Enhanced dissipation}. 
 In~\cite{constantinkiselevryzhikzlatos}, Constantin, Kiselev, Ryzhik, and Zlato{\v s} investigated enhanced dissipation in broad generality in terms of spectral properties of the operator $\boldsymbol{u} \cdot \nabla$. This work was recently revisited and quantified by Coti Zelati, Delgadino, and Elgindi in~\cite{ontherelationshipcpam} and Wei in~\cite{dongyiweiresolvent}. Shear flows are not relaxation enhancing in the sense of~\cite{constantinkiselevryzhikzlatos}.

Shear flows are steady solutions of the two-dimensional Euler equations and (forced) Navier-Stokes equations. The corresponding linearized equations are
\begin{equation}
    \label{eq:linearizedns}
    \p_t \omega + b(y) \p_x \omega - b''(y) \p_x \psi= \nu \Delta \omega, \quad \Delta \psi = \omega,
\end{equation}
where $\omega$ is the vorticity perturbation. Compared to~\eqref{eq:fpde1},~\eqref{eq:linearizedns} contains the additional non-local term $- b''(y) \p_x \psi$ which complicates the analysis. This term vanishes for the Couette flow $b(y) = y$, and we can extract the inviscid damping and enhanced dissipation directly from the solution formula in~\eqref{eq:couetteexactsol}. Notable papers on nonlinear stability of the Couette flow include~\cite{bedrossianmasmoudiinviscid,bmv1,bedrossianvicolwang,ionescujia}. Other important flows are the Poiseuille (pipe/channel) flow $b(y) = y(1-y)$~\cite{michelepoiseuille}, Kolmogorov (sinusoidal) flow~\cite{beckwaynebar,kolmogorovhypocoercivitywei,kolmogorovwei}, and Oseen vortices~\cite{gallaghergallaynier,wendeng1,wendeng2,weioseen,gallayaxisymmetrization}.


When considering passive scalars in an infinite pipe, one refers to  \emph{Taylor--Aris dispersion}~\cite{taylor1954dispersion,aris1956dispersion}. This topic was recently reexamined via center manifolds and hypocoercivity  in~\cite{beckchaudharywayne}.

Passive scalars in power law circular flows were considered by Dolce and Coti Zelati in~\cite{michelesquared}.

Passive scalars in shear flows with non-local diffusion were recently investigated by He~\cite{siminghefractional} by resolvent methods.

\emph{2. Hypoellipticity}. H{\"o}rmander's theorem was reproved using pseudodifferential operators by Kohn in~\cite{Kohnsproof}, and it is Kohn's proof which H{\"o}rmander includes in his monograph~\cite{hormandervol3}, see Chapter~22. We mention also the stochastic proof reviewed in~\cite{Hairermalliavin}. Otherwise, the literature is immense, and we refer to~\cite{bramanti} for a survey on H{\"o}rmander's vector fields and the books~\cite{helffernier,lernerbook} on microlocal analysis. 

\emph{3. Kinetic equations and regularity techniques}.  Recent work on kinetic equations has illuminated connections between H\"ormander's hypoellipticity and regularity techniques in the style of De Giorgi--Nash--Moser.  In~\cite{silvestre1}, Silvestre derived \emph{a priori} H\"older estimates for the Boltzmann equation, and follow-up work of Silvestre and Imbert~\cite{imbertsilvestre1} obtained weak Harnack and H\"older estimates for a large class of integro-differential equations, including the Boltzmann equation without cutoff. Golse, Imbert, Mouhot, and Vasseur~\cite{gimv} studied the Harnack inequality for kinetic Fokker-Planck and Landau equations. For open questions and a presentation of these and related programs of research, we refer to Mouhot~\cite{mouhot1}. 

\subsection*{Acknowledgements}
DA was supported by the NSF Postdoctoral Fellowship Grant No. 2002023. RB was supported by the NSF Graduate Fellowship Grant No. 1839302. 
MN was supported by the National Science Foundation under Grant No. DMS-1928930 while participating in a program hosted by the Mathematical Sciences Research Institute in Berkeley, California, during the spring 2021 semester. The authors also thank Vlad Vicol, Scott Armstrong, Jean--Christophe Mourrat, and Jacob Bedrossian for valuable discussions.

\section{Preliminaries}



Following the convention specified in the introduction, the torus $\T$ may be regarded as the set $(0,1)$ with periodic boundary conditions. Furthermore, unless specified otherwise, functions $f$ and $u$ are \emph{qualitatively smooth}, though they are only measured in low regularity.

\begin{definition}[Flow Maps and Pullbacks]\label{def:flow:maps}
Let $E\subseteq D$ be open and $(t_1,t_2)$ be given. Let $X$ be a Lipschitz vector field defined on $\T\times E\times (t_1,t_2)$. For  $\tilde E\subseteq E$ open, $(\tilde t_1,\tilde t_2)\subseteq(t_1,t_2)$, and $\sigma \in (\sigma_1,\sigma_2)$, consider the function
\begin{equation}\notag
\phi(\sigma,x,y,t)\:(\sigma_1,\sigma_2)\times \T\times \tilde E\times (\tilde t_1,\tilde t_2) \to \T \times E \times (t_1,t_2) \, 
\end{equation}
defined as the solution to the ordinary differential equation\footnote{In making this definition, we have implicitly assumed that the flow remains contained in $\T\times E\times(t_1,t_2)$, a condition which will be satisfied for every vector field considered in this paper.}
\begin{equation}\notag
\partial_\sigma \phi = X(\phi), \qquad \phi(0,x,y,t) = (x,y,t) \, .
\end{equation}
Then for each $\sigma\in(\sigma_1,\sigma_2)$ and any measurable function $u:\Omega\times(0,1)\rightarrow \R$, we may define the action of the associated pull back $\exp(\sigma X)$ on $u$ by 
\begin{equation}\notag
\exp (\sigma X) u (x,y,t) = u\left(\exp(\sigma X) (x,y,t) \right) = u\left(\phi(\sigma,x,y,t)\right) \, .
\end{equation}
\end{definition}

Throughout the paper, we shall use the following notation and shorthand for particular vector fields.

\begin{definition}[Vector Fields]\label{def:vector:fields}
We fix notation for the following vector fields.  $A$ and $B$ are allowed to be arbitrary Lipschitz vector fields as defined in Definition~\ref{def:flow:maps}, and $b\in C^\infty(D)$ satisfies \eqref{eq:crit:points}.
\begin{subequations}
\begin{align}
    [A,B] &=  (\ad A) B = AB - BA \label{eq:comm:defn}\\
    Y &= \nu^{\sfrac{1}{2}} \p_y \\
    X_0 &= \p_t + b(y) \p_x \\
    Z_1 &= [Y, X_0] = \nu^{\sfrac{1}{2}} b'(y) \p_x  \label{eq:z:k}\\
    Z_{k} &= [Y, Z_{k-1}] = \nu^{\sfrac{k}{2}} b^{(k)}(y) \p_x \, , \qquad  k\geq 2 \, . \label{eq:z:kplusone}
\end{align}
\end{subequations}
\end{definition}

We shall frequently refer to the flow maps and pull backs induced by the vector fields defined above.  For example, in the case of $Y$ and $X_0$, these flows maps are simply translations along the trajectories of $Y$ and $X_0$, defined by
\begin{equation}\notag
    \exp(\sigma Y)u(x,y,t) = u(x,y+\sigma \nuhalf , t), \qquad \exp(\sigma X_0)u(x,y,t) = u(x+\sigma b(y), y, t+\sigma) \, .
\end{equation}

\newcommand{\czero}{c_0}
\newcommand{\ck}{c_k}
\newcommand{\ckminus}{c_{k-1}}
\newcommand{\ckplus}{c_{k+1}}
\newcommand{\cfin}{c_{N+1}}
\begin{definition}[Parameters]
For $0\leq k \leq N+2$, we shall use the notation
\begin{equation}\label{eq:sk:def}
s_k = \frac{1}{k+2}, \qquad m_k=\frac{1}{s_k} \, .
\end{equation}
Note that by the definition of $m_k$, we have $m_{k+1}=1+m_k$.  In addition, for $0\leq k \leq N+1$, we define an increasing sequence of positive numbers
\begin{equation}\label{eq:ck:def}
    0 < \czero < c_1 < \cdots < \ckminus < \ck < \ckplus < \cdots < \cfin < \infty \, .
\end{equation}
While we shall specify the exact values of the $\ck$ later,\footnote{{See conditions \eqref{eq:ck:one}, \eqref{eq:ck:brak:one}, \eqref{eq:ck:brak:two}, \eqref{eq:ck:brak:three}, and \eqref{eq:ck:brak:combo}, which culminate in the choice of the $c_k$'s in Corollary~\ref{cor:improvedbrackets}. }}  each $\ck$ satisfies $\ck^{2} \approx_{b,k} \nu^{-\frac{N+1}{N+3}}$, and $\czero^2 = \nu^{-\frac{N+1}{N+3}}$.
\end{definition}

We shall use $s_k$, $m_k$, and $\ck$ to define a number of Besov-type norms. The parameter $s_k$ specifies an \emph{index of differentiability}, while $\ck$ quantifies the \emph{magnitudes} of the finite differences allowed in each Besov-type norm. These spaces are complicated by (1) the fact that they are $B^s_{p,q}$-type Besov norms with $q=+\infty$, (2) finite differences must stay within the space-time domain under consideration, and (3) an extension procedure would be non-trivial (for example, $X_0 u$ cannot be extended by a simple even reflection in time).

\begin{definition}[Besov Spaces]\label{def:besov:spaces}
We define the following fractional Besov-type seminorm along the vector fields $X_0$: 
\begin{align}
    \norm{u}_{\X} = \sup_{\sigma \in (0,\sfrac{\czero}{\sqrt{2}})} \sigma^{-1} \big{(}& \norm{\exp (\sigma^{2} X_0) u - u}_{L^2(\Omega \times (0,\sfrac{c_0^2}{{2}}))}\notag\\
    &\qquad + \norm{\exp(-\sigma^2 X_0) u - u}_{L^2(\Omega \times (\sfrac{\czero^2}{2},\czero^2))} \big{)} \notag
\end{align}
For $u \in L^2_{t,x} H^1_y(\Omega \times (0,c_0^2))$ (qualitative assumption), we define the seminorm
\begin{equation}\notag
    \| u \|_{\Y} = \nu^{\sfrac{1}{2}} \| \p_y u \|_{L^2(\Omega \times (0,c_0^2))} \, .
\end{equation}
This norm is appropriate in Dirichlet, periodic, or Neumann settings, and by the analysis in Evans p. 292--293 \cite{evans} is equivalent to
\begin{align}
    \tilde{\|} u \|_{\Y} = \sup_{\sigma \in \left(0,\frac{\nu^{-\sfrac{1}{2}}}{2}\right)} \sigma^{-1} \big{(}& \norm{\exp (\sigma Y) u - u}_{L^2(\T \times (0,\sfrac{1}{2}) \times (0,\czero^2))}\notag\\
    &\qquad + \norm{\exp(-\sigma Y) u - u}_{L^2(\T \times (\sfrac{1}{2},1) \times (0,\czero^2))} \big{)} \, . \notag
\end{align}
That is, $\| u \|_{\Y} \approx \tilde{\|} u \|_{\Y}$. For $g \in L^2_{t,x} (H^1)^*_y(\Omega \times (0,c_0^2))$ ($H^{-1}_y := (H^1_0)^*_y$ in the Dirichlet setting), we define the dual seminorm
\begin{equation}\notag
    \left\| g \right\|_{\Y^*(\Omega\times(0,\czero^2))} = \sup_{ \left\| u \right\|_{\Y}=1 } \int_{\Omega \times (0,c_0^2)} gu \, dx \, dy \, dt  < \infty \, 
\end{equation}
where $u \in L^2_{t,x} H^1_y(\Omega \times (0,c_0^2))$ ($(H^1_0)_y$ in the Dirichlet setting). This definition is suitable for the integration by parts in Section~\ref{sec:proofofmaintheorem}.

Consider a vector field $Z = c(y) \p_x$ and $s>0$. For a free parameter $\sigma_Z>0$, we define
\begin{equation}\label{eq:Qz:def}
	\norm{u}_{Q^s_Z} = \sup_{\sigma \in (0,\sigma_Z)} \sigma^{-1} \norm{\exp \sigma^{\sfrac{1}{s}} Z u - u}_{L^2(\Omega \times (0,\czero^2))} \, .
\end{equation}
If $c(y)\equiv 1$, we set {$\sigma_{\p_x}=\infty$}\footnote{Due to the periodicity in $x$, finite differences for $\sigma> 1$ cannot contribute to the supremum.} so that
\begin{equation}\label{eq:q:s:dx:def}
\left\| u \right\|_{Q^s_{\p_x}} = \sup_{\sigma\in(0,\infty)} \sigma^{-1} \left\| \exp(\sigma^{\sfrac{1}{s}} \p_x) u - u \right\|_{L^2(\Omega\times(0,\czero^2))} \, .
\end{equation}
For $k\geq 1$, set {$\sigma_{Z_k}=\ck$}, the precise value of which will be determined later.  Then define
\begin{equation}\notag
    \norm{u}_{\bZ_k} = \norm{u}_{Q^{s_k}_{Z_k}(\Omega \times (0,\czero^2))} = \sup_{(0,\ck)} \sigma^{-1} \left\| \exp(\sigma^{m_k}Z_k) u - u \right\|_{L^2(\Omega\times(0,\czero^2))} \, .
\end{equation}
\end{definition}

The $Q$ stands for `quotient'. Let us make an analogy for such spaces in the periodic single variable case $x\in \T$. We have that $Q^1_{\p_x}$ and $\dot H^1$ are equivalent \cite{evans}. When $s \in (0,1)$ is not an integer, we have that $Q^s_{\p_x}$ and $\dot B^s_{2,\infty}$ are equivalent.
\smallskip

Before recording a few simple properties of these spaces, let us first set some notation. Recalling \eqref{eq:crit:points}, let $U_i$ be a sufficiently small neighborhood of $y_i$ so that for $k\geq 2$,
\begin{equation}\notag
    B_{k} = \inf_{y\in \mathcal{U}_k} \left|b^{(k)}(y)\right| > 0 \, , \qquad \mathcal{U}_k = \bigcup_{\{i:N_i=k-1\}} U_i \, ,
\end{equation}
and
\begin{equation}\notag
    B_{1} = \inf_{y\in \mathcal{U}_1} \left|b'(y)\right| >0 \,, \qquad \mathcal{U}_1 = (0,1)\setminus\bigcup_{i=1}^M U_i \, .
\end{equation}
The next lemma asserts that $\mathcal{U}_k$ is the set of $y$-values on which $Z_k$ is coercive in the $x$-direction, and $B_k$ is the lower bound on the coercivity.

\begin{lemma}[Comparable norms]\label{lem:comparable}
Let $A(y):D\rightarrow \R$ be a smooth function, and let $\mathcal{U}_k$ and $B_k$ be defined as above.  
\begin{enumerate}
\item If $A(y)$ satisfies $ 0 \leq A \leq A_1$, then for any value of $\sigma_{A\p_x}$, we have that
\begin{equation}
\label{eq:A:comparable}
 	 \norm{u}_{Q_{A \p_x}^s} \les_s A_1^s \norm{u}_{Q^s_{\p_x}} \, .
\end{equation}
If in addition ${\sigma_{A\p_x}\geq A_0^{-s}\geq 1}$ and $0 < A_0 \leq A \leq A_1$, we have the two-sided bound
\begin{equation}
\label{eq:A:comparable:alt}
 	 A_0^s \norm{u}_{Q^s_{\p_x}} \les_s \norm{u}_{Q_{A \p_x}^s} \les_s A_1^s \norm{u}_{Q^s_{\p_x}} \, .
\end{equation}
\item For $k\geq 1$ and
\begin{equation}\label{eq:ck:one}
{\ck\geq\left(\nu^{\sfrac{k}{2}}B_k\right)^{-s_k}} \geq 1 \, ,
\end{equation}
we have that
\begin{equation}
\label{eq:Zk:comparable}
    \left(\nu^{\sfrac{k}{2}} B_k \right)^{s_k} \left\| \mathbf{1}_{\mathcal{U}_k} u \right\|_{Q_{\p_x}^{s_k}} \leq \left\| \mathbf{1}_{\mathcal{U}_k} u \right\|_{\bZ_k} \leq \left(\nu^{\sfrac{k}{2}} \left\| b^{(k)} \right\|_{L^\infty(D)} \right)^{s_k} \left\| u \right\|_{Q_{\p_x}^{s_k}} \, .
\end{equation}
\end{enumerate}
\end{lemma}
\begin{proof}
Since time plays a passive role in this argument, we shall suppress the time dependence. We first prove \eqref{eq:A:comparable} following Hormander's Lemma 4.1 \cite{hormander67}.  After making the identification $\tilde\sigma=\sigma^{\sfrac{1}{s}}$ in \eqref{eq:Qz:def}, we may write that
\begin{align}
\label{eq:increment:bound}
\tilde\sigma^{-2s}&\| \exp(\tilde\sigma A(y) \p_x ) u - u \|_{L^2(\Omega\times(0,\czero^2))}^2 \notag\\
&= \tilde\sigma^{-2s} \barint_{0\leq h\leq \tilde\sigma} \| \exp(\tilde\sigma A(y) \p_x ) u - u \|_{L^2(\Omega\times(0,\czero^2))}^2  dh \notag\\
&\leq \tilde\sigma^{-2s} \barint_{0\leq h \leq \tilde\sigma} \int_{\Omega\times(0,\czero^2)} |u (x + \tilde\sigma A(y), y) - u(x+hA_1,y) |^2\, dx\, dy\,dt\, dh \notag \\
&\qquad + \tilde\sigma^{-2s} \barint_{0\leq  h \leq \tilde\sigma} \int_{\Omega\times(0,\czero^2)} |u (x + hA_1, y) - u(x,y) |^2\, dx\,dy\,dt\, dh \notag \\
&\leq \tilde\sigma^{-2s} \barint_{0\leq h \leq \tilde\sigma} \int_{\Omega\times(0,\czero^2)} |u (x + \tilde\sigma A(y), y) - u(x + h A_1 ,y) |^2 \, dx\,dy\,dt\, dh  + A_1^{2s} \| u \|_{Q_{\p_x}^s}^2 \, .
\end{align}
In order to estimate the first integral, we define the new variables 
\begin{equation}
\label{eq:change:of:variables}
    x' = x + h A_1, \qquad h = \tilde h + \tilde\sigma \frac{A(y)}{A_1} \, .
\end{equation}
Using \eqref{eq:change:of:variables}, we have that
\begin{align}
\label{eq:increment:bound:2}
   &  \tilde\sigma^{-2s} \barint_{0\leq h \leq \tilde\sigma} \int_{\Omega\times(0,\czero^2)} |u (x + \tilde\sigma A(y), y) - u(x+hA_1,y) |^2\, dx\, dy\,dt\, dh \notag\\
   & \leq \tilde\sigma^{-2s} \cdot 2 \barint_{ -\tilde\sigma\sfrac{A}{A_1} \leq \tilde h \leq \tilde \sigma} \int_{\Omega\times(0,\czero^2)} |u (x'-\tilde h A_1, y) - u(x' ,y) |^2\, dx'\, dy\,dt\, d\tilde h \notag \\
   &\leq \tilde\sigma^{-2s} \cdot 2 \barint_{ -\tilde\sigma\sfrac{A}{A_1} \leq \tilde h \leq \tilde \sigma} A_1^{2s} \tilde{h}^{2s} \left\| u \right\|_{Q_{\p_x}^s} \notag \\
   &\lesssim A_1^{2s} \left\| u \right\|_{Q_
  {\p_x}^s}^2 \, .
\end{align}
Combining \eqref{eq:increment:bound} and \eqref{eq:increment:bound:2} we have that
\begin{equation}
\label{eq:A:upperbound}
    \| u \|_{Q_{A \p_x}^s }^2 \lesssim A_1^{2s} \| u \|_{Q_{ \p_x}^s }^2 \, ,
\end{equation}
which proves \eqref{eq:A:comparable}.

To prove the lower bound in \eqref{eq:A:comparable:alt}, we use that one may assume that $\sigma\in(0,1)$ in \eqref{eq:q:s:dx:def} and again make the identification $\tilde\sigma=\sigma^{\sfrac{1}{s}}$ to write that 
\begin{align}
    A_0^{2s} \left\| u \right\|_{Q_{\p_x}^s}^2 &= A_0^{2s} \sup_{\tilde\sigma\in(0,1)} \tilde\sigma^{-2s} \barint_{0\leq h \leq \tilde\sigma A_0^{-1}} \int_{\Omega\times(0,\czero^2)} | u(x+\tilde\sigma,y) - u(x,y) |^2 \,dx\,dy\,dt\,dh \, . \notag
\end{align}
Adding and subtracting $u(x+hA(y),y)$ this time gives that
\begin{align}
    \tilde\sigma^{-2s} &\barint_{0\leq h \leq \tilde\sigma A_0^{-1}} \int_{\Omega\times(0,\czero^2)} | u(x+\tilde\sigma,y) - u(x,y) |^2 \,dx\,dy\,dt\,dh \notag\\
    &\les \tilde\sigma^{-2s}\barint_{0\leq h \leq \tilde \sigma A_0^{-1}} \int_{\Omega\times(0,\czero^2)} | u(x+\tilde\sigma,y) - u(x+hA(y),y) |^2 \,dx\,dy\,dt\,dh \notag\\
    &\qquad + \tilde\sigma^{-2s}\barint_{0\leq h \leq \tilde \sigma A_0^{-1}} \int_{\Omega\times(0,\czero^2)} |  u(x+hA(y),y) -u(x,y) |^2 \,dx\,dy\,dt\,dh \, . \label{eq:splitting:one}
\end{align}
To bound the second term, we make the change of variables $h=\tilde h^{\sfrac{1}{s}}$ and obtain
\begin{align}
    \tilde\sigma^{-2s}&\barint_{0\leq h \leq \tilde \sigma A_0^{-1}} \int_{\Omega\times(0,\czero^2)} |  u(x+hA(y),y) -u(x,y) |^2 \,dx\,dy\,dt\,dh \notag\\
    &=\tilde\sigma^{-2s}\cdot (\tilde\sigma A_0^{-1})^{-1} \cdot \frac{1}{s} \int_{0\leq \tilde h \leq (\tilde\sigma A_0^{-1})^{s}} \int_{\Omega\times(0,\czero^2)} |u(x+\tilde h^{\sfrac{1}{s}}A(y),y) - u(x,y)|^2 \,dx\,dy\,dt \, \tilde{h}^{\sfrac{1}{s}-1} \,d\tilde h \notag\\
    &\leq \tilde\sigma^{-2s-1} \cdot A_0 \cdot\frac{1}{s} \int_{0\leq \tilde h \leq (\tilde\sigma A_0^{-1})^{s}} \tilde{h}^{\sfrac{1}{s}+1} \left\| u \right\|_{Q_{A\p_x}^s}^2 \,d\tilde h \notag \\
    &\les_s A_0^{-2s} \left\| u \right\|_{Q_{A\p_x}^s}^2 \,  \notag
\end{align}
which gives the desired bound in \eqref{eq:A:comparable:alt}. Notice that in order to appeal to the $\left\| u \right\|_{Q_{A\p_x}^s}$ norm, we have used that $\tilde h \leq (\tilde\sigma A_0^{-1})^s\leq A_0^{-s} \leq \sigma_{A\p_x}$.

For the first term from \eqref{eq:splitting:one}, we use the change of variables
\begin{equation}
\label{eq:change:of:variables:2}
    x' = x + A(y)h, \qquad h = h'+ \frac{\tilde\sigma}{A(y)} \,,
\end{equation}
which is well defined by the assumption that $A(y)\geq A_0 >0$. Then the first term from \eqref{eq:splitting:one} may be estimated as
\begin{align}
    &\tilde\sigma^{-2s}\barint_{0\leq h \leq \tilde \sigma A_0^{-1}} \int_{\Omega\times(0,\czero^2)} | u(x+\tilde\sigma,y) - u(x+hA(y),y) |^2 \,dx\,dy\,dt\,dh \notag\\
    &= \tilde\sigma^{-2s} \int_{\Omega\times(0,\czero^2)} \barint_{{\tilde\sigma}{A(y)}^{-1}\leq h' \leq \tilde\sigma A_0^{-1}} | u(x'-h'A(y),y) - u(x',y) |^2 \,dh' \,dx'\,dy\,dt \notag  \\
    &\leq 2 \tilde\sigma^{-2s-1} A_0 \cdot\frac{1}{s} \int_{\Omega\times(0,\czero^2)} \int_{|\tilde h|\leq A_0^{-s}} | u(x'-\tilde h^{\sfrac{1}{s}}A(y),y) - u(x',y) |^2 \tilde h^{\sfrac{1}{s}-1} \,d\tilde h \,dx'\,dy\,dt \notag \\
    &\leq 2 \tilde\sigma^{-2s-1} A_0 \cdot\frac{1}{s} \int_{|\tilde h|\leq A_0^{-s}} \tilde h^{\sfrac{1}{s}+1} \left\| u \right\|_{Q_{A\p_x}^s}^2 \,d\tilde h \notag\\
    &\les_s A_0^{-2s} \left\| u \right\|_{Q_{A\p_x}^s}^2 
\end{align}
Note that after making the change of variables $\tilde h^{\sfrac{1}{s}}=-h'$ above, we have again used the assumption that $|\tilde h|\leq A_0^{-s}\leq \sigma_{A\p_x}$ to appeal to the $\left\| u \right\|_{Q_{A\p_x}^s}$ norm. 

Finally, the proof of \eqref{eq:Zk:comparable} follows from \eqref{eq:A:comparable:alt} after using the definition of $\mathcal{U}_k$, $B_k$, and $Z_k$.
\end{proof}

We shall use the following fractional Poincar{\'e} and interpolation inequalities. 

\begin{lemma}[Fractional Poincar{\'e} and Interpolation]\label{lem:fractional:poincare}
The following two estimates hold:
\begin{enumerate}
\item If $u \in Q^{s}_{\p_x}$  and $0< s'<s<1$, then adopting the usual notation for $\langle u \rangle(y,t)$, we have that\footnote{For a definition of Banach-space valued Sobolev functions, one may for example refer to Amann~\cite{amann97}.}
\begin{equation}
	\norm{u - \langle u \rangle(y,t)}_{L^2(\Omega\times(0,\czero^2))} + \norm{u}_{L^2\left((0,\czero^2);\dot{H}^{s'}\left(\mathbb{T};L^2_y\right)\right)} \les_{s,s'} \norm{u}_{Q^{s}_{\p_x}} \, . \notag
\end{equation}
\item If $0 \leq s_1 < s_2 \leq 1$, $0< \theta < 1$, and $s=\theta s_1+\theta s_2$, then
\begin{align}
    \left\| u \right\|_{Q^s_{\partial_x}} &\lesssim \left\| u \right\|_{Q^{s_1}_{\partial_x}}^\theta  \left\| u \right\|_{Q^{s_2}_{\partial_x}}^{1-\theta} \notag\\
    \left\| u \right\|_{\dot{H}^s\left(\mathbb{T};L^2_y\right)} &\lesssim \left\| u \right\|_{\dot{H}^{s_1}\left(\mathbb{T};L^2_y\right)}^\theta  \left\| u \right\|_{\dot{H}^{s_2}\left(\mathbb{T};L^2_y\right)}^{1-\theta} \, . \notag
\end{align}
\end{enumerate}
\end{lemma}
\begin{proof}
To prove the first inequality, we appeal to Jensen's inequality, the lower bound on $|x-\tilde x|^{-1-s}$ for $x,\tilde x\in \mathbb{T}$, and the bound on $u$ in $Q^s_{\partial_x}$ to write that
\begin{align}
    \left\| u(x,y,t)-\langle u \rangle(y,t) \right\|_{L^2(\Omega\times(0,\czero^2))}^2 &= \iiint_{\mathbb{T}\times(0,1)\times(0,\czero^2)} \left| u(x,y,t) - \barint_{\mathbb{T}} u(\tilde{x},y,t) \,d\tilde x \right|^2 \, dy \, dx \, dt  \notag\\
    & \lesssim \iiint_{\mathbb{T}\times(0,1)\times(0,\czero^2)} \barint_{\mathbb{T}}\frac{|u(x,y,t)-u(\tilde x,y,t)|^2}{|x-\tilde x|^{1+2s'}} \,d\tilde x\,dy \,dx \, dt \notag\\
    &\lesssim \iint_{\mathbb{T}\times(0,\czero^2)} \int_{-\tilde x}^{1-\tilde x} \frac{\left\| u(\tilde x+\sigma,y,t)-u(\tilde x,y,t) \right\|_{L^2_y}^2}{|\sigma|^{1+2s'}} \, d\sigma \, d \tilde x \,dt \notag\\
    &\lesssim \int_0^{\czero^2}\int_{-1}^{1} \frac{\left\| u(\tilde x+\sigma,y,t)-u(\tilde x,y,t) \right\|_{L^2_{y,\tilde x}}^2}{|\sigma|^{1+2s'}} \, d\sigma\,dt  \notag\\
    &\lesssim \| u \|_{Q^s_{\partial_x}}^2 \int_{-1}^{1} \frac{|\sigma|^{2s}}{|\sigma|^{1+2s'}} \,d\sigma \notag\\
    &\lesssim \| u \|_{Q^s_{\partial_x}}^2 \, .
\end{align}
The implicit constant depends only on $s,s'\in(0,1)$ and in particular is independent of $u$. Notice that the second line is precisely the norm of $u$ in $L^2\left((0,1);H^{s'}\left(\mathbb{T};L^2_y\right)\right)$, concluding the proof.

When $s_1>0$, the third inequality follows from writing
\begin{align}
    \frac{\left| u(x,y)-u(\tilde x,y) \right|^2}{|x-\tilde x|^{1+2s}} =     \frac{\left| u(x,y)-u(\tilde x,y) \right|^{2\theta}}{|x-\tilde x|^{\theta(1+2s_1)}}\frac{\left| u(x,y)-u(\tilde x, y) \right|^{2(1-\theta)}}{|x-\tilde x|^{(1-\theta)(1+2s_2)}},
\end{align}
where $\tilde{x} = x + \sigma$,
and then applying H\"{o}lder's inequality with $L^\frac{1}{\theta}$ and $L^\frac{1}{1-\theta}$.  The case $s_1=0$ and the second inequality are similar, and we omit further details.
\end{proof}

\section{Bracket inequalities}\label{sec:hormander}
%

The arguments in this section are based on the following equality, which may be checked directly. To streamline calculations, we adopt the convention, locally in Section~\ref{sec:hormander}, that $Z_0=X_0$, and $\left\| u \right\|_{\bZ_0} = \left\| u \right\|_{\X}$.  Then for $k\geq 0$,
\begin{equation}\label{eq:expansion:basic}
\begin{aligned}
	&\exp (-\varepsilon\sigma^{m_k} Z_k) \exp (- \varepsilon^{-1}\sigma Y) \exp (\varepsilon\sigma^{m_k} Z_k  ) \exp (\varepsilon^{-1}\sigma Y) u(x,y,t) - u(x,y,t) \\
	&\quad = u\left(x + \varepsilon\sigma^{m_k} \nu^{\sfrac{k}{2}} \left[b^{(k)}(y+\varepsilon^{-1}\sigma\nuhalf) - b^{(k)}(y)\right], y,t\right) - u(x,y,t) \, . 
	\end{aligned}
\end{equation}
Using the heuristic that
$$ \varepsilon\sigma^{m_k} \nu^{\sfrac{k}{2}} \left[ b^{(k)}(y+\varepsilon^{-1}\sigma\nuhalf)-b^{(k)}(y) \right]\approx \sigma^{m_{k+1}}\nu^{\frac{k+1}{2}} b^{(k+1)}(y) \, $$
and ignoring terms from the Taylor expansion which have been hidden inside the $\approx$ symbol, we would expect to control $\left\| u \right\|_{{\bZ_{k+1}}}$ using $\left\| u \right\|_{\Y}$ and $\left\| u \right\|_{\bZ_{k}}$. The leftover terms in the Taylor expansion will in fact be controlled by ``higher brackets," meaning $\left\| u \right\|_{\bZ_{j}}$ for $j\geq k+2$. The $\varepsilon$ accounts for the fact that the $\bZ_k$ and $\Y$ norms allow for finite differences of variable magnitude and helps us rebalance inequalities to absorb remainder terms.


\begin{lemma}[Bracket inequality]\label{lemma:bracket}
If $0 \leq k \leq N$ and $\varepsilon_{k+1}$ satisfies the technical condition \eqref{eq:ck:brak:combo}, then for an implicit constant independent of the choice of $\varepsilon_{k+1}$, we have the estimate
\begin{align}
	 \norm{u}_{\bZ_{k+1}} &\les \varepsilon_{k+1}^{s_k} \norm{u}_{\bZ_{k}} + \varepsilon_{k+1}^{-1} \norm{u}_{\Y} +  \sum_{j=k+2}^{N+1} \varepsilon_{k+1}^{s_j(k-j+1)} \norm{u}_{\bZ_{j}} \notag\\
	 &\qquad +  \left(\varepsilon_{k+1}^{k-N-1}\nu^{\frac{N+2}{2}} \norm{b^{(N+2)}}_{L^\infty}\right)^{s_{N+2}} \norm{u}_{Q^{s_{N+2}}_{\p_x}}\, . \label{eq:bracket:modified}
\end{align}
\end{lemma}
\begin{remark}
We check in Corollary~\ref{cor:improvedbrackets} that such a choice of $\varepsilon_{k+1}$ is possible, and then we fix values for $\varepsilon_{k+1}$ and also $c_{k+1}$.
\end{remark}

Notice that we assume control over the norm $\left\| u \right\|_{\X}$, which we show in Section~\ref{sec:interpolation} is controlled by the \emph{a priori} estimates enjoyed by solutions to~\eqref{eq:fpde}.

\begin{proof}[Proof of Lemma~\ref{lemma:bracket}]
To ensure that the finite differences stay inside the domain, which has non-trivial boundary in $t$ as well as in $y$ for the Dirichlet and Neumann cases, we subdivide $\Omega$ into regions $\T \times I \times J$, where $I \in \{ (0,\sfrac{1}{{2}}), (\sfrac{1}{2},1) \}$ and $J \in \{ (0,\sfrac{\czero^2}{{{2}}}), (\sfrac{\czero^2}{{2}},\czero^2) \}$. We only write the case when $I = (0,\sfrac{1}{2})$ and $J=(0,\sfrac{\czero^2}{2})$. The cases with $I = (\sfrac{1}{2},1)$ follow upon replacing $Y$ by $-Y$, and the cases with $J = (\sfrac{\czero^2}{2},\czero^2)$ follow upon replacing $X$ by $-X$. We divide the remainder of the proof into steps.  In the first step, we write out the Taylor expansion corresponding to \eqref{eq:expansion:basic}.  In the second step, we use the admissible ranges of $\sigma$ to derive the constraints \eqref{eq:ck:brak:one}, \eqref{eq:ck:brak:two}, and \eqref{eq:ck:brak:three} on $\varepsilon_{k+1}$ in terms of $\ck$.  Finally, the last step bounds each term using the appropriate Besov norm. The parameter $k$ is fixed throughout the rest of the proof, and we simply write $\varepsilon$ rather than $\varepsilon_{k+1}$.
\smallskip

\textit{Step 1. \, } From Taylor's theorem, we have that
\begin{align}\notag
    b^{(k)}(y+\varepsilon^{-1}\sigma\nuhalf) &= b^{(k)}(y) + \varepsilon^{-1}\sigma\nuhalf b^{(k+1)}(y)\notag\\
    &\qquad + \sum_{j=k+2}^{N+1} \frac{(\varepsilon^{-1}\sigma\nuhalf)^{j-k}}{(j-k)!} b^{(j)}(y) \, + \, O\left((\varepsilon^{-1}\sigma\nuhalf)^{N-k+2}\right) \, , \notag
\end{align}
where the prefactor on $O\left((\varepsilon^{-1}\sigma\nuhalf)^{N-k+2}\right)$ is bounded from above by $\left\| b^{(N+2)} \right\|_{L^\infty}$. Multiplying through by $\varepsilon\sigma^{k+2}\nu^{\sfrac{k}{2}}$ and rearranging, it follows that
\begin{align}
    \sigma^{k+3} \nu^{\frac{k+1}{2}} b^{(k+1)}(y) &= \varepsilon\sigma^{k+2}\nuhalfk\left(b^{(k)}(y+\varepsilon^{-1}\sigma\nuhalf)-b^{(k)}(y)\right) \notag\\
    &\qquad - \sum_{j=k+2}^{N+1} \frac{\sigma^{j+2}\nu^{\sfrac{j}{2}}}{(j-k)!} \varepsilon^{k-j+1} b^{(j)}(y) \, + \, O\left(\sigma^{N+4}\right) \, , \notag
\end{align}
where now the prefactor on $O(\sigma^{N+4})$ is bounded from above by $\varepsilon^{k-N-1}\nu^\frac{N+2}{2}\left\| b^{(N+2)} \right\|_{L^\infty}$. From \eqref{eq:sk:def} and \eqref{eq:expansion:basic}, it follows that (recall that when $k=0$, $Z_0=X_0$)
\begin{equation}
	\label{eq:expansion}
\begin{aligned}
	\exp (\sigma^{m_{k+1}}  Z_{k+1} ) u - u = \exp (-\varepsilon\sigma^{m_k} Z_k) \exp (-\varepsilon^{-1}\sigma Y) \exp (\varepsilon
	\sigma^{m_k}  Z_k  ) \exp (\varepsilon^{-1}\sigma Y) R u - u \, ,
	\end{aligned}
\end{equation}
where the remainder operator $R$ is defined by
\begin{equation}\notag
\begin{aligned}
	R u = \left( \prod_{j = k+2}^{N+1} \exp\left(-\frac{\varepsilon^{k-j+1}}{(j-k)!}{\sigma^{m_j} Z_{j}}\right) \right) \exp ( H \p_x ) u \, ,
	\end{aligned}
\end{equation}
and $H(y,\sigma) = O(\sigma^{m_{N+2}}) $ with prefactor bounded above by $\varepsilon^{k-N-1}\nu^\frac{N+2}{2}\left\| b^{(N+2)} \right\|_{L^\infty}$. Notice that each of the above transformations is volume preserving, but the transformations induced by $Z_j$, $j \geq 1$, and $H$ preserve the domain $\T \times I \times J$, while the transformations induced by $X_0$ and $Y$ do not and will be estimated separately.
\smallskip

\textit{Step 2. \, }
The left-hand side of \eqref{eq:expansion} must be measured in $L^2$ for $\sigma\in(0,\ckplus)$, which necessitates that the same range of $\sigma$ is admissible in the right-hand side. Since all admissible values of $\sigma$ are positive, we may track the largest value of $\sigma$, which for $Z_{k+1}$ is $\ckplus$. In order for $\exp(\varepsilon^{-1}\sigma Y)$ to be well defined, we need
\begin{equation}\label{eq:ck:brak:one}
    \varepsilon^{-1} \ckplus \leq \frac{\nu^{-\frac{1}{2}}}{2} \, .
\end{equation}
In order for $\exp(\varepsilon \sigma^{m_k}Z_k)$ to be well defined, we need
\begin{equation}\label{eq:ck:brak:two}
    \varepsilon\ckplus^{m_k} \leq \ck^{m_k} \, .
\end{equation}
We require that the term with $\exp(H)u$ must be well defined for $\sigma\in(0,c_{k+1})$. However, this term will be bounded using the $Q_{\p_x}^{s_{N+2}}$ norm, for which any values of $\sigma$ are admissible. The other terms in the remainder for $k+2\leq j \leq N+1$ require that
\begin{equation}\label{eq:ck:brak:three}
    \frac{\varepsilon^{k-j+1}}{(j-k)!} \ckplus^{m_j} \leq c_j^{m_j} \, .
\end{equation}
Since the first and third conditions provide lower bounds on $\varepsilon$ and the second provides an upper bound, we consolidate the three conditions by imposing that
\begin{equation}\label{eq:ck:brak:combo}
     \max\left(2\ckplus\nu^{\sfrac{1}{2}}, \max_{k+2\leq j \leq N+1} \left(\left(\frac{\ckplus}{c_j}\right)^{m_j} \frac{1}{(j-k)!} \right)^{\frac{1}{j-1-k}}\right) \leq  \varepsilon \leq \left(\frac{c_k}{c_{k+1}}\right)^{m_k} \, .
\end{equation}
Notice that the right hand side involves an upper bound in terms of $\ck$ and $\ckplus$, while the left hand side involves lower bounds with powers of $c_j$ in the denominator where $j\geq k+2$, and so we shall pick our sequence of `$c_k$'s in increasing order with respect to $k$.
\smallskip

\textit{Step 3. \, } To estimate the left-hand side of \eqref{eq:expansion} in $L^2(\T \times I \times J)$, we introduce a telescoping sum on the right-hand side and apply the triangle inequality:
\begin{equation}
\begin{aligned}\notag
 \big{\|}\exp (\sigma^{m_{k+1}}  &Z_{k+1} ) u - u\big{\|}_{L^2}  \\
 &\leq \big{\|}\exp (-\varepsilon\sigma^{m_k} Z_k) \exp (-\varepsilon^{-1}\sigma Y) \exp (\varepsilon\sigma^{m_k}  Z_k  ) \exp (\varepsilon^{-1}\sigma Y) R u \notag\\
 &\qquad\qquad\qquad\qquad -  \exp (-\varepsilon^{-1}\sigma Y) \exp (\varepsilon\sigma^{m_k}  Z_k  ) \exp (\varepsilon^{-1}\sigma Y) R u\big{\|}_{L^2} \\
 &\quad + \norm{ \exp (-\varepsilon^{-1}\sigma Y) \exp (\varepsilon\sigma^{m_k}  Z_k  ) \exp (\varepsilon^{-1}\sigma Y) R u -   \exp (\varepsilon\sigma^{m_k}  Z_k  ) \exp (\varepsilon^{-1}\sigma Y) R u}_{L^2} \\
 &\quad +\norm{ \exp (\varepsilon\sigma^{m_k}  Z_k  ) \exp (\varepsilon^{-1}\sigma Y) R u -  \exp (\varepsilon^{-1}\sigma Y) R u}_{L^2}\\
 &\quad + \norm{\exp (\varepsilon^{-1}\sigma Y) R u -  R u}_{L^2} + \norm{Ru - u}_{L^2} \, .
\end{aligned}
\end{equation}
Let $\iota = \mathbf{1}_{k = 0}$. 
Changing variables in the integrals on the right-hand side of the previous inequality yields
\begin{equation}
\begin{aligned}
 &\norm{\exp (\sigma ^{m_{k+1}}  Z_{k+1} ) u - u}_{L^2(\T \times I \times J)} \notag\\
 &\quad\leq \norm{\exp (-\varepsilon\sigma^{m_k} Z_k) u - u}_{L^2(\T \times I \times (\varepsilon\iota \sigma^{m_k} + J))} + \norm{ \exp (-\varepsilon^{-1}\sigma Y)  u - u}_{L^2\left(\T \times (\varepsilon^{-1}\nuhalf \sigma + I) \times (\varepsilon\iota \sigma^{m_k} + J)\right)} \\
  &\qquad + \norm{ \exp (\varepsilon\sigma^{m_k}  Z_k  )  u - u}_{L^2\left(\T \times (\varepsilon^{-1}\nuhalf \sigma + I) \times J\right)} + \norm{  \exp (\varepsilon^{-1}\sigma Y) u -  u}_{L^2(\T \times I \times J)} \notag\\
  &\qquad\quad + \norm{ Ru - u}_{L^2(\T \times I \times J)} \, .
\end{aligned}
\end{equation}
Hence,
\begin{equation}
      \norm{\exp (\sigma ^{m_{k+1}}  Z_{k+1} ) u - u}_{L^2(\T \times I \times J)}\les \norm{ Ru - u}_{L^2(\T \times I \times J)} + \sigma \varepsilon^{s_k} \norm{u}_{\bZ_k} + \sigma \varepsilon^{-1} \norm{u}_{\Y} \, . \label{eq:combine:1}
\end{equation}
To estimate $\norm{ Ru - u}_{L^2(\T \times I \times J)}$, we again introduce a telescoping sum and change variables:
\begin{equation}\label{eq:combine:2}
\begin{aligned}
    \norm{ Ru - u}_{L^2(\T \times I \times J)} &\lesssim \sum_{j=k+2}^{N+1} \norm{ \exp\left(-\frac{\varepsilon^{k-j+1}}{(j-k)!}\sigma^{m_j} Z_{j}\right) - u}_{L^2(\T \times I \times J)} + \norm{\exp(H\p_x) u - u}_{L^2(\T \times I \times J)} \\
   & \les \sum_{j=k+2}^{N+1} \sigma \varepsilon^{s_j(k-j+1)} \norm{u}_{\bZ_{j}} + \sigma \left(\varepsilon^{k-N-1}\nu^{\frac{N+2}{2}} \norm{b^{(N+2)}}_{L^\infty}\right)^{s_{N+2}} \norm{u}_{Q^{s_{N+2}}_{\p_x}}.
    \end{aligned}
\end{equation}
Here, we have applied estimate \eqref{eq:A:comparable} from Lemma~\ref{lem:comparable} with $s = s_{N+2}$, $A(y)=h(y)\sigma^{m_{N+2}}$ for some unknown function $h(y)$ depending on the exact form of the remainder $H(y,\sigma)$, and $A_1 = \nu^{\frac{N+2}{2}} \left\| b^{(N+2)} \right\|_{L^\infty}$ to estimate the term containing $H\p_x$. Combining \eqref{eq:combine:1} and \eqref{eq:combine:2} completes the proof.
\end{proof}

\begin{corollary}[Improved bracket inequality, choice of $c_k$, and conditions on $\nu_0$]
\label{cor:improvedbrackets}
For all $0 \leq k \leq N+1$, $\varepsilon>0$, and $\nu \ll_{b} 1$,
\begin{equation}
	\label{eq:improvedbrackets}
	 \norm{u}_{\bZ_k} \leq \varepsilon \norm{u}_{\X} + C(\varepsilon,b)\norm{u}_{\Y} + C(\varepsilon,b)\left(\nu^{\frac{N+2}{2}} \norm{b^{(N+2)}}_{L^\infty}\right)^{s_{N+2}} \norm{u}_{Q^{s_{N+2}}_{\p_x}} \, .
\end{equation}
\end{corollary}
\begin{proof}
The proof proceeds in steps according to the value of a dummy parameter $\ell\in\{1,..., N+1\}$, so that when $\ell$ reaches the terminal value $N+1$, we derive several conditions on $\nu_0$ and conclude the proof. We begin by choosing {$c_0=\nu^{-\frac{N+1}{2(N+3)}}$}.
\smallskip

\textit{Step 1. \,} Choose $c_1$ so that \eqref{eq:ck:one} is satisfied for $k=1$. Now choose $\varepsilon_1$ in \eqref{eq:bracket:modified} (level $k=0$) so that 
\begin{align}\notag
	 \norm{u}_{\bZ_{1}} &\leq \sfrac{\varepsilon}{2} \norm{u}_{\X} + C(\varepsilon,b) \bigg{[} \norm{u}_{\Y} +  \sum_{j=2}^{N+1}  \norm{u}_{\bZ_{j}} + \left(\nu^{\frac{N+2}{2}} \norm{b^{(N+2)}}_{L^\infty}\right)^{s_{N+2}} \norm{u}_{Q^{s_{N+2}}_{\p_x}} \bigg{]} \, ,
\end{align}  
and so that \eqref{eq:ck:brak:two} is satisfied at level $k=0$. This choice of $\varepsilon_1$ may be made \emph{independently of $\nu$} as long as we have chosen $c_1$ to be a $\nu$-independent multiple of $c_0$. As a result of the choices of $c_0$, $c_1$, and $\varepsilon_1$, we have additional new lower bounds on all $c_m$ for $m\geq 2$ from \eqref{eq:ck:brak:three}. Note that our choices and the resulting additional lower bounds may depend on $b$ and $\varepsilon$, but are independent of $\nu$.
\smallskip

\textit{Step 2. \,} Choose $c_2$ so that the new lower bound from \eqref{eq:ck:brak:three} is satisfied for $j=2$ and $k=0$, and so that \eqref{eq:ck:one} is satisfied for $k=2$. Now choose $\varepsilon_2$ in \eqref{eq:bracket:modified} so that for $k=1,2$,
\begin{align}\notag
	 \norm{u}_{\bZ_{k}} &\leq \sfrac{3\varepsilon}{4} \norm{u}_{\X} + C(\varepsilon,b) \bigg{[} \norm{u}_{\Y} +  \sum_{j=3}^{N+1}  \norm{u}_{\bZ_{j}} + \left(\nu^{\frac{N+2}{2}} \norm{b^{(N+2)}}_{L^\infty}\right)^{s_{N+2}} \norm{u}_{Q^{s_{N+2}}_{\p_x}} \bigg{]} \, ,
\end{align} 
and so that \eqref{eq:ck:brak:two} is satisfied at level $k=1$. Notice that the sum crucially begins at $j=3$ in the above inequality. As a result of the choices of $\varepsilon_2$ and $c_2$, we have additional new lower bounds on all $c_m$ for $m\geq 3$ from \eqref{eq:ck:brak:three} which are $\nu$-independent, provided that $c_2$ is a $\nu$-independent multiple of $c_1$.
\smallskip

\textit{Step $k$. \,} Assume that Steps $1,...,k-1$ have been completed successfully, thus fixing choices of $c_0,...,c_{k-1}$ and $\varepsilon_1,...,\varepsilon_{k-1}$. Choose $c_k$ so that all lower bounds resulting from choices of $c_{k'}$ for $k'<k$ from \eqref{eq:ck:brak:three} are satisfied, and so that \eqref{eq:ck:one} is satisfied at level $k$.  Now choose $\varepsilon_k$ in \eqref{eq:bracket:modified} so that for $\tilde k =1,2,...,k$, 
\begin{align}\notag
	 \norm{u}_{\bZ_{\tilde k}} &\leq \frac{(2^{\tilde k}-1)\varepsilon}{2^{\tilde k}} \norm{u}_{\X} + C(\varepsilon,b) \bigg{[} \norm{u}_{\Y} +  \sum_{j=k+1}^{N+1}  \norm{u}_{\bZ_{j}} + \left(\nu^{\frac{N+2}{2}} \norm{b^{(N+2)}}_{L^\infty}\right)^{s_{N+2}} \norm{u}_{Q^{s_{N+2}}_{\p_x}} \bigg{]} \, ,
\end{align}
and so that \eqref{eq:ck:brak:two} is true at level $k-1$. The sum in the above inequality now begins at $k+1$, and we have additional new lower bounds an all $c_m$ for $m\geq k+1$ from \eqref{eq:ck:brak:three}.
\smallskip

\textit{Conclusion. \,} Upon reaching $k=N+1$, the sum in the above inequality is empty, and we thus conclude the proof after ensuring that $\nu$, and by extension $\nu_0$, is small enough so that \eqref{eq:ck:brak:one} is satisfied for all $k$. Note that this is possible since $c_0=\nu^{-\frac{N+1}{2(N+3)}}\ll \nu^{-\frac
{1}{2}}$, and each subsequent $c_k$ was chosen so that \eqref{eq:ck:one}, \eqref{eq:ck:brak:two}, and \eqref{eq:ck:brak:three} were satisfied. The latter two conditions are $\nu$-independent and the first depends on $\nu^{-\frac{ks_k}{2}}$, which is increasing in $k$ and bounded above by $\nu^{-\frac{N+1}{2(N+3)}}$.
\end{proof}

\begin{proposition}[Subelliptic estimate, Poincar{\'e} inequality, and choice of $\nu_0$]\label{prop:subelliptic}
	For all $\nu\leq \nu_0 \ll_b 1$ and $\varepsilon > 0$, we have that
\begin{equation}
\label{eq:subellipticestimate}
	\nu^{\frac{N+1}{2(N+3)}} \left( \norm{u- \langle u \rangle }_{L^2(\Omega\times(0,\czero^2))} + \norm{u}_{Q^{s_{N+1}}_{\p_x}} \right) \les_b \varepsilon\norm{u}_{\X} + C(\varepsilon)\norm{u}_{\Y} \, ,
\end{equation}
where the implicit constant is independent of $\nu$ and $\varepsilon$ but depends on $b$.
\end{proposition}
\begin{proof}
From Lemma~\ref{lem:comparable} and Lemma~\ref{lem:fractional:poincare}, we have that for $k\geq 1$,
\begin{equation}\notag
	\norm{ \mathbf{1}_{\mathcal{U}_k} u }_{Q^{s_{k}}_{\p_x}} \les \left(\nu^{\sfrac{k}{2}} B_k \right)^{-s_k} \norm{ \mathbf{1}_{\mathcal{U}_k} u }_{\bZ_{k}} \, .
\end{equation}
Since $\nu^{-\sfrac{ks_k}{2}}\leq \nu^{-\sfrac{(k+1)s_{k+1}}{2}}$ for all $k\geq 1$, we obtain using Corollary~\ref{cor:improvedbrackets} that 
\begin{equation}
\begin{aligned}
 \norm{u}_{Q^{s_{N+1}}_{\p_x}} & \les_b \nu^{-\frac{N+1}{2(N+3)}} \sum_{k=1}^{N+1} \norm{u}_{\bZ_{k}} \\
	&\les_b \varepsilon \norm{u}_{\X} + C(\varepsilon,b)\norm{u}_{\Y} + C(\varepsilon,b)\left(\nu^{\frac{N+2}{2}} \norm{b^{(N+2)}}_{L^\infty}\right)^{s_{N+2}} \norm{u}_{Q^{s_{N+2}}_{\p_x}} \, . \label{eq:sub:one}
	\end{aligned}
\end{equation}
We then apply the fractional Poincar\'e inequality from Lemma~\ref{lem:fractional:poincare} to $u-\langle u \rangle$ to deduce that
\begin{equation}\label{eq:sub:two}
    \left\| u - \langle u \rangle \right\|_{L^2} \les_b \varepsilon \norm{u}_{\X} + C(\varepsilon,b)\norm{u}_{\Y} + C(\varepsilon,b)\left(\nu^{\frac{N+2}{2}} \norm{b^{(N+2)}}_{L^\infty}\right)^{s_{N+2}} \norm{u}_{Q^{s_{N+2}}_{\p_x}} \, .
\end{equation}
Applying the fractional interpolation inequality from Lemma~\ref{lem:fractional:poincare} to the $\left\| u \right\|_{Q_{\p_x}^{s_{N+2}}}$ term on the right-hand side of \eqref{eq:sub:one} and \eqref{eq:sub:two} and using that
\begin{equation}\notag
   C(\varepsilon,b)\left(\nu^{\frac{N+2}{2}} \norm{b^{(N+2)}}_{L^\infty}\right)^{s_{N+2}} \ll_\nu 1
\end{equation}
if $\nu \leq \nu_0 \ll_b 1$ concludes the proof.  Note that we have chosen $\nu_0$ so that the above inequality is satisfied and the conditions on $\nu_0$ from Corollary~\ref{cor:improvedbrackets} are satisfied.

\end{proof}

\section{Interpolation inequality}\label{sec:interpolation}

The estimates in the previous section relied on an iteration written in terms of norms with positive indices of differentiability.  However, the \emph{a priori} estimate for $\left(\partial_t + b(y) \partial_x \right)u$ from~\eqref{eq:fpde} is in the space $\Y^*$, which in the Dirichlet-in-$y$ case one may identify with $L^2_{t,x}H^{-1}_y$ up to a $\nu$-dependent factor.  Combined with the estimate for $u$ in $L^2_{t,x}H^1_y$, one may heuristically expect, by interpolation, that ``$\left(\partial_t + b(y) \partial_x \right)^{\sfrac{1}{2}}u \in L^2_{t,x} L^2_y$. "  We will show that such an interpolation inequality is possible by following H\"ormander's arguments in Section 5 of~\cite{hormander67}. Note that Corollary~\ref{cor:improvedbrackets} and Proposition~\ref{prop:subelliptic} control a number of norms of $u$ by a large constant times $\left\| u \right\|_{\Y}$ and a small constant times $\left\| u \right\|_{\X}$, and so it will be acceptable to use these quantities when estimating $\left\| u \right\|_{\X}$. 

\begin{definition}[Mollification]\label{mollifying}
For $1\leq k\leq N+1$, let $Z_k$ be defined as in \eqref{eq:z:k} and \eqref{eq:z:kplusone}, and $s_k$ and $m_k$ be defined as in \eqref{eq:sk:def}.
\begin{enumerate}
    \item Let $\phi \in C_0^\infty(-1,1)$ be a smooth, even, non-negative function with unit $L^1$ norm. We set
\begin{align}\label{eq:phi:k}
    \phi_{\eta,k} u (x,y,t) &= \int_\R \left(\exp\left(\eta^{m_k}\sigma Z_k\right)u\right)(x,y,t) \phi(\sigma) \, d\sigma \, \notag\\
    &= \int_\R u(x+\eta^{m_k}\sigma \nu^{\sfrac{k}{2}}b^{(k)}(y),y,t)\phi(\sigma)\, d\sigma \, .
\end{align}
Note that since $\phi$ is even, $\phi_{\eta,k}$ is self-adjoint on $L^2(\Omega\times(0,\czero^2))$.
    \item Let $\Phi \in C_0^\infty(-1,1)$ be a smooth, even, non-negative function with unit $L^1$ norm. We set
\begin{align}\label{eq:PHI}
    \Phi_{\eta} u (x,y,t)&= \int_{\R} (\exp(\eta^{m_{N+2}}x' \nu^{\frac{N+2}{2}} \| b^{(N+2)} \|_{L^\infty} \partial_x))u(x,y,t) \Phi(x') \, dx' \notag\\
    &= \int_\R u\left(x+\eta^{m_{N+2}}x' \nu^{\frac{N+2}{2}} \| b^{(N+2)} \|_{L^\infty} ,y,t\right) \Phi(x') \,dx' \,
\end{align}
and note that $\Phi_{\eta}$ is self-adjoint on $L^2(\Omega\times(0,\czero^2))$ as well.
    \item For $\eta>0$ and $1\leq j \leq N+1$, define $\phi_{\eta,\geq j} u$ by
    \begin{equation}\label{eq:S:eta}
        \phi_{\eta,\geq j} u = \prod_{k=j}^{N+1} \phi_{\eta,k} \Phi_{\eta} u \, . 
    \end{equation}
When $j=1$, we simply write $\phi_\eta u$, and when $j=N+2$, we define $\phi_{\eta,\geq j} u = \Phi_\eta u$.
\end{enumerate}
\end{definition}
\begin{remark}
Since the smoothing operators are mollifications-in-$x$ for each fixed $t$ and $y$, they actually commute. Moreover, we have
\begin{equation}\label{eq:moll:1}
    [X_0, \phi_{\eta,k}] u = \left[ \exp(\tau X_0), \phi_{\eta,k} \right] u = 0 \, ,
\end{equation}
with the same equalities holding for $\phi_{\eta,\geq j}$ and $\phi_\eta$.

\end{remark}

\begin{lemma}[Commutator Identity]\label{taylorseriesidentities}
    Let $Z$, $Z_0$ be smooth vector fields, and recall from \eqref{eq:comm:defn} that $(\ad Z_0)Z$ denotes the vector field $[Z_0,Z]=Z_0Z - ZZ_0$. Let $M \in \N_0$. Define
    \begin{equation}
        \label{eq:adjointrepdef}
        \exp(-\tau\, \ad Z_0)Z u = \exp(-\tau Z_0) Z \exp(\tau Z_0)u(x,y,t)
    \end{equation}
    Then we have the identity
     \begin{align}
      \exp(-\tau\, \ad Z_0)Z u = \sum_{k=0}^M \frac{(-\tau)^k}{k!} \left( \ad Z_0 \right)^k Z u(x,y,t) + \int_0^\tau (\ad  Z_0)^{M+1} Z u(x,y,t) (\tau-s)^{M} \, ds  \, . \label{eq:barf:identity:1}
        \end{align}
\end{lemma}

\begin{proof}
Consider the formal identity
\begin{equation}
    \exp(-\tau Z_0) Z \exp(\tau Z_0)u = \sum_{k=0}^\infty \frac{(-\tau)^k}{k!} (\ad Z_0)^k Z u,
\end{equation}
which can be `checked' by differentiating the left-hand side $k$ times with respect to $\tau$, evaluating at $\tau=0$, and comparing with the Taylor series on the right-hand side. Since the series may not converge, we instead truncate the Taylor series to order $M$, which, after using the integral form of the remainder term, leads directly to~\eqref{eq:barf:identity:1}. 
\end{proof}

\begin{lemma}[Mollification Estimates]\label{lem:moll}
With $\phi_\eta$ defined as in Definition~\ref{mollifying}, the following estimates hold.
\begin{enumerate}
\item Let $\eta^2\leq \sfrac{c_0^2}{2}$ be given.  Then
\begin{equation}\label{eq:moll:2}
    \left\| \phi_\eta u - u \right\|_{L^2(\Omega \times (0,\czero^2))} \lesssim \eta \left( \sum_{k=1}^{N+1} \left\|u\right\|_{\bZ_k} + \left\| u \right\|_{Q^{s_{N+2}}_{\partial_x}}\right) \, .
\end{equation}
\item For $0\leq\tau\leq\eta^2 \leq \sfrac{\czero^2}{{2}}$,
\begin{align}
    &\left\| \exp(\tau X_0) \exp(-\tau \ad X_0) Y
    \phi_\eta \phi_\eta u \right\|_{L^2(\Omega \times (0,\sfrac{\czero^2}{2}))}  \notag\\
    &\qquad \qquad \lesssim \left\| u \right\|_{\Y} + \sum_{k=1}^{N+1} \left\|u\right\|_{\bZ_k} + \left( \nu^{\frac{N+2}{2}} \left\| b^{(N+2)} \right\|_{L^\infty} \right)^{s_{N+2}} \left\| u \right\|_{{Q}^{s_{N+2}}_{\partial_x}}  \, , \label{eq:moll:3}
\end{align}
with a similar estimate holding if $\tau$ is replaced with $-\tau$ and the time interval $(0,\sfrac{\czero^2}{2})$ is replaced with $(\sfrac{\czero^2}{2},\czero^2)$.
\end{enumerate}
\end{lemma}

The estimate~\eqref{eq:moll:2} is used to control the terms~\eqref{eq:firsttermofthisthing} and~\eqref{eq:thirdtermofthisthing} in the proof of Proposition~\ref{lemma:Z0:estimate} (Interpolation inequality), whereas~\eqref{eq:moll:3} is used to control the more involved term~\eqref{eq:secondtermofthisthing}.

\begin{proof}[Proof of~\eqref{eq:moll:2}]
We may write $\phi_\eta u - u$ in a telescoping sum as
\begin{align}
    \phi_\eta u - u &= \prod_{k=1}^{N+1} \phi_{\eta,k} \Phi_\eta u - u \notag\\
    &= \sum_{m=1}^{N+1} \left( \prod_{k=1}^{m} \phi_{\eta,k} u\ - \prod_{k=1}^{m-1} \phi_{\eta,k} u \right) + \prod_{k=1}^{N+1}\phi_{\eta,k}\Phi_\eta u - \prod_{k=1}^{N+1}\phi_{\eta,k} u \notag \, .
\end{align}
Using again that convolution is a bounded operator from $L^2(\Omega\times(0,\czero^2))$ to itself, it suffices to estimate in $L^2(\Omega\times(0,\czero^2))$ the quantities
\begin{equation}\notag
    \phi_{\eta,k}u - u, \qquad \Phi_\eta u - u
\end{equation}
for $1\leq k \leq N+1$. This is achieved by writing
\begin{align}
    &\left\| \phi_{\eta,k}u(x,y,t) - u(x,y,t) \right\|^2_{L^2(\Omega\times(0,\czero^2))}\notag\\
    &\qquad \leq \int_0^{\czero^2} \int_0^1 \int_\T \int_{\R} \left( \exp (\sigma \eta ^{m_k} Z_k)u(x,y,t) - u(x,y,t) \right)^2 \phi(\sigma) \, d\sigma \,dx\,dy\,dt \, \notag\\
    &\qquad = \int_0^{\czero^2} \int_0^1 \int_\T \int_{\R} \left( \exp ((\tilde\sigma \eta)^{m_k} Z_k)u(x,y,t) - u(x,y,t) \right)^2 \phi(\tilde\sigma^{m_k}) m_k \tilde\sigma^{m_k-1} \, d\tilde\sigma \,dx\,dy\,dt \, \notag\\
    &\qquad \les \eta^2 \left\| u \right\|^2_{\bZ_k} \, , \notag
\end{align}
with a similar estimate for $\Phi_\eta$. Note that since $\phi$ is compactly supported in $(-1,1)$, we may assume $\tilde\sigma=\sigma^{s_k}$ is contained in the interval $(0,1)$, and so $\exp((\tilde\sigma\eta)^{m_k}Z_k)u$ is well defined since $\tilde\sigma\eta\leq\eta\leq\sfrac{c_0}{\sqrt{2}}\leq c_k$ for all $k\geq 1$ by \eqref{eq:ck:def}. \end{proof}

\begin{proof}[Proof of~\eqref{eq:moll:3}]
To make our situation more concrete, we exploit the identity
\begin{align}\notag
\exp(-\tau \ad X_0) Y = Y + \tau Z_1 = \nu^{\sfrac{1}{2}}\partial_y + \tau \nu^{\sfrac{1}{2}} b'(y) \partial_x \, ,
\end{align}
which does not include any translations in time.  Since $\exp(\tau X_0)$ consists of a volume-preserving diffeomorphism in the $x$ variable and a shift in the time variable, it will suffice to estimate
\begin{equation}\notag
    \left\| \exp(-\tau \ad X_0) Y \phi_\eta \phi_\eta u \right\|_{L^2(\Omega\times(0,\czero^2))} = \left\| \left(Y + \tau Z_1 \right) \phi_\eta \phi_\eta u \right\|_{L^2(\Omega\times(0,\czero^2))}\, .
\end{equation}

\textit{Step 1. Estimating $Y\phi_\eta\phi_\eta u$}. To commute $Y$ past the mollifications, we will demonstrate
\begin{equation}\label{eq:comm:estimate:j}
 \left\| [Y,\phi_{\eta,\geq j}\phi_{\eta,\geq j}] u \right\|_{L^2(\Omega\times(0,\czero^2))} \lesssim   \left\| u \right\|_{\Y} + \sum_{k=1}^{N+1} \left\|u\right\|_{\bZ_k} + \left( \nu^{\frac{N+2}{2}} \left\| b^{(N+2)} \right\|_{L^\infty} \right)^{s_{N+2}} \left\| u \right\|_{Q^{s_{N+2}}_{\partial_x}} 
\end{equation}
using decreasing induction on $j$ and beginning with the case $j=N+2$. Since $Yu$ is clearly bounded in $L^2$ by the right-hand side of \eqref{eq:comm:estimate:j}, this will then produce an estimate for $Y\phi_\eta\phi_\eta u$ which matches \eqref{eq:moll:3}.

Beginning with the case $j=N+2$, recall from \eqref{eq:S:eta} that $\phi_{\eta,\geq N+2}=\Phi_\eta$. Then since $Y$ differentiates in $y$ and $\Phi_\eta$ mollifies in $x$, $Y$ in fact commutes with both operators, demonstrating the base case of the induction.  Therefore, we may assume that $1\leq j \leq N+1$ and that \eqref{eq:comm:estimate:j} holds for $j'>j$.  We calculate that
\begin{align}
    \left[ Y, \phi_{\eta,\geq j}\phi_{\eta,\geq j} \right] u &= \left[ Y, \phi_{\eta,j}\phi_{\eta,j} \phi_{\eta,\geq j+1} \phi_{\eta,\geq j+1} \right] u \notag\\
    &=\phi_{\eta,j} \phi_{\eta,j} [Y,\phi_{\eta,\geq j+1} \phi_{\eta,\geq j+1}] u + \left[ Y, \phi_{\eta,j} \phi_{\eta,j} \right] \phi_{\eta,\geq j+1} \phi_{\eta,\geq j+1} u \,. \notag
\end{align}
Due to the induction hypothesis, we may immediately bound the first term, so we focus on the second term. We may expand this term using a modification of Lemma~\ref{taylorseriesidentities} with $Z_0=Z_j$ and $Z=Y$ in which we have applied $\exp(\eta^{m_j}(\sigma-\sigma^*)Z_j)$ to both sides: 
\begin{align}
    &Y \iint_{\R^2} \exp\left(\eta^{m_j}\left(\sigma-\sigma^*\right) Z_j \right) \phi_{\eta,\geq j+1}\phi_{\eta,\geq j+1} u(x,y,t) \, \phi(\sigma) \phi(\sigma^*) \, d\sigma d\sigma^* \notag\\
    &\qquad \qquad \qquad -\iint_{\R^2} \exp\left(\eta^{m_j}\left(\sigma-\sigma^*\right) Z_j\right) Y \phi_{\eta,\geq j+1}\phi_{\eta,\geq j+1} u(x,y,t) \,  \phi(\sigma) \phi(\sigma^*) \, d\sigma d\sigma^* \notag\\
    &\overset{\eqref{eq:adjointrepdef}}{=}\iint_{\R^2} \bigg{[} \exp\left(\eta^{m_j}\left(\sigma-\sigma^*\right) Z_j\right) \exp\left(-\eta^{m_j}\left(\sigma-\sigma^*\right) \ad Z_j\right) Y \phi_{\eta,\geq j+1} \phi_{\eta,\geq j+1} u(x,y,t) \notag\\
    &\qquad \qquad \qquad - \exp\left(\eta^{m_j}\left(\sigma-\sigma^*\right) Z_j\right) Y \phi_{\eta,\geq j+1} \phi_{\eta,\geq j+1} u(x,y,t) \bigg{]} \phi(\sigma)\phi(\sigma^*) \, d\sigma d\sigma^* \notag\\
    &\overset{\eqref{eq:barf:identity:1}}{=} \iint_{\R^2} \exp\left(\eta^{m_j}(\sigma-\sigma^*) Z_j\right) \notag\\
    &\qquad\qquad \times \sum_{k=1}^\infty \frac{(-\eta)^{km_j}(\sigma-\sigma^*)^k}{k!} \left( \ad Z_j \right)^k Y \phi_{\eta,\geq j+1} \phi_{\eta,\geq j+1} u(x,y,t) \, \phi(\sigma)\phi(\sigma^*) \, d\sigma d\sigma^* \, . \label{eq:comm:mess:3}
\end{align}
Notice that the infinite series representation is well defined since
\begin{align}
    -\left(\ad Z_j \right)Y = Z_{j+1},  \qquad \left(\ad Z_j \right)^2 Y = -\left( \ad Z_j \right) Z_{j+1} = -\left[ \nu^\frac{j}{2} b^{(j)}(y) \partial_x , \nu^\frac{j+1}{2} b^{(j+1)}(y) \partial_x \right] = 0 \, , \notag
\end{align}
and so the only non-zero term in \eqref{eq:comm:mess:3} is the term in the series for $k=1$. 

It will suffice to obtain an $L^2$ bound for
\begin{align}\label{eq:to:bound}
Z_{j+1} \phi_{\eta,\geq j+1} \phi_{\eta,\geq j+1} u(x,y,t) \,   \end{align}
for $1\leq j \leq N+1$.  In fact, we will estimate a version of this term for $j\geq 0$, as the $j=0$ case will be required shortly. We begin by relabeling $j$ with $0\leq j\leq N+1$ as $j$ with $1\leq j\leq N+2$ in \eqref{eq:to:bound}. Then we analyze how the derivatives `hit' the mollifiers: For $1\leq j< N+2$, we have\footnote{The notation $Z_j \cdot u(\cdots)$ below emphasizes that the differential operator $Z_j$ is applied to the object $u(\cdots)$, in contrast to $(Z_j u)(\cdots)$.}
\begingroup
\allowdisplaybreaks
\begin{align}
    &Z_j \phi_{\eta,\geq j} \phi_{\eta, \geq j} u(x,y,t) \notag\\
    &= Z_j \prod_{k=j}^{N+1} \phi_{\eta,k} \Phi_\eta \prod_{k=j}^{N+1} \phi_{\eta,k} \Phi_\eta \, u(x,y,t) \notag\\
    &= \int_{\R^{2(N+3-j)}} \Phi(x') \Phi(x^*) \prod_{k=j}^{N+1} \phi(\sigma_k)\phi(\sigma_k^*) \notag\\
    &\quad \times Z_j \cdot u\left(x+\nu^{\frac{N+2}{2}} \| b^{(N+2)} \|_{L^\infty} \eta^{{m_{N+2}}}(x'-x^*)+\sum_{k=j}^{N+1}\nu^{\sfrac{k}{2}} b^{(k)}(y) \eta^{m_k}(\sigma_k-\sigma_k^*),y,t\right) \,dx' \,dx^* \, \prod_{k=j}^{N+1} d\sigma_k\,d\sigma_k^* \notag\\
    &= \int_{\R^{2(N+3-j)}} \Phi(x') \Phi(x^*) \left( \prod_{k=j}^{N+1} \phi(\sigma_k)\phi(\sigma_k^*) \right) \nu^{\sfrac{j}{2}}b^{(j)}(y) \notag\\
    &\quad \times \partial_x u\left(x+\nu^{\frac{N+2}{2}} \| b^{(N+2)} \|_{L^\infty} \eta^{{m_{N+2}}}(x'-x^*)+\sum_{k=j}^{N+1}\nu^{\sfrac{k}{2}} b^{(k)}(y) \eta^{m_k}\left(\sigma_k-\sigma_k^*\right),y,t\right) \,dx' \,dx^* \, \prod_{k=j}^{N+1} d\sigma_k\,d\sigma_k^* \, . \notag
    \end{align}
\endgroup
We need to multiply this expression by $\eta^{m_{j-1}}$ and then bound the resulting object by the right hand side of \eqref{eq:moll:3}. Towards this end, we first convert the $\nu^{\sfrac{j}{2}} b^{(j)}(y)\partial_x$ to $\eta^{-m_j}\left(\partial_{\sigma_j}-\partial_{\sigma_j^*}\right)$, yielding
\begingroup
\allowdisplaybreaks
    \begin{align}
    &\eta^{m_{j-1}-m_j} \int_{\R^{2(N+3-j)}} \left(\partial_{\sigma_j}-\partial_{\sigma_j^*}\right) \cdot  u\left(x+\nu^{\frac{N+2}{2}} \| b^{(N+2)} \|_{L^\infty} \eta^{{m_{N+2}}}(x'-x^*)+\sum_{k=j}^{N+1}\nu^{\sfrac{k}{2}} b^{(k)}(y) \eta^{m_k}\left(\sigma_k-\sigma_k^*\right),y,t\right)\notag\\
    &\qquad \qquad \qquad \times \Phi(x') \Phi(x^*) \left( \prod_{k=j}^{N+1} \phi(\sigma_k)\phi(\sigma_k^*) \right) \,dx' \,dx^* \, \prod_{k=j}^{N+1} d\sigma_k\,d\sigma_k^* \notag \\
    &\overset{\footnotesize\rm IBP}{=} -\eta^{-1} \int_{\R^{2(N+3-j)}}  u\left(x+\nu^{\frac{N+2}{2}} \| b^{(N+2)} \|_{L^\infty} \eta^{{m_{N+2}}}(x'-x^*)+\sum_{k=j}^{N+1}\nu^{\sfrac{k}{2}} b^{(k)}(y) \eta^{m_k}\left(\sigma_k-\sigma_k^*\right),y,t\right)\notag\\
    &\qquad \qquad \qquad \times \Phi(x') \Phi(x^*) \left(\partial_{\sigma_j}-\partial_{\sigma_j^*}\right) \left( \prod_{k=j}^{N+1} \phi(\sigma_k)\phi(\sigma_k^*) \right) \,dx' \,dx^* \, \prod_{k=j}^{N+1} d\sigma_k\,d\sigma_k^* \notag\\
    &=-\eta^{-1} \int_{\R^{2(N+3-j)}} \bigg{[} u\left(x+\nu^{\frac{N+2}{2}} \| b^{(N+2)} \|_{L^\infty} \eta^{{m_{N+2}}}(x'-x^*)+\sum_{k=j}^{N+1}\nu^{\sfrac{k}{2}} b^{(k)}(y) \eta^{m_k}\left(\sigma_k-\sigma_k^*\right),y,t\right)\notag\\
    &\qquad \qquad \qquad \qquad - u\left(x+\nu^{\frac{N+2}{2}} \| b^{(N+2)} \|_{L^\infty} \eta^{{m_{N+2}}}(x'-x^*)+\sum_{k=j+1}^{N+1}\nu^{\sfrac{k}{2}} b^{(k)}(y) \eta^{m_k}\left(\sigma_k-\sigma_k^*\right),y,t\right) \bigg{]} \notag\\
    &\qquad \qquad \qquad \times \Phi(x') \Phi(x^*) \left(\partial_{\sigma_j}-\partial_{\sigma_j^*}\right) \left( \prod_{k=j}^{N+1} \phi(\sigma_k)\phi(\sigma_k^*) \right) \,dx' \,dx^* \, \prod_{k=j}^{N+1} d\sigma_k\,d\sigma_k^* \, . \label{eq:mess:case1}
\end{align}
\endgroup
Notice that the second-to-last term, which may be subtracted for free due to the fact that we have differentiated with respect to $\sigma_j-\sigma_j^*$, crucially gives a finite difference along $Z_j$ in the $x$ variable. Now since $\sigma_k-\sigma_k^*$ is at most of order one, we have a finite difference of order $\eta^{m_j}$ multiplied by $\eta^{-1}$, and so we can then apply Jensen's inequality and integrate in $L^2(\Omega\times(0,\czero^2))$ to control the $L^2$ norm of the right-hand side of~\eqref{eq:mess:case1} by $\left\|u\right\|_{\bZ_j}$.

It remains to treat the case $j=N+2$, in which we must bound $Z_{N+2} \Phi_\eta \Phi_\eta u$.  The argument is similar, and we omit further details. 
\smallskip

\textit{Step 2. Estimating $\tau Z_1 \phi_\eta\phi_\eta u$}.
Notice that $ Z_1 \phi_\eta \phi_\eta $ is precisely what we calculated in the previous step for $j=1$. The only difference in the calculation is that instead of multiplying by a power of $\eta$, we multiply by $\tau$.  But by the assumption that $\tau\leq \eta^2$, we pick up a prefactor of $\eta^{-1}$ exactly as in \eqref{eq:mess:case1}.  We omit further details.

\end{proof}

\begin{proposition}
[Interpolation inequality]\label{lemma:Z0:estimate}
For a constant independent of $u$, we have the estimate
\begin{equation}\label{eq:z0:estimate}
\left\| u \right\|_{\X} \les \left\| u \right\|_{\Y} + \left\|  X_0 u\right\|_{\Y^*} + \sum_{k=1}^{N+1} \left\|u\right\|_{\bZ_k} + \left( \nu^{\frac{N+2}{2}} \norm{b^{(N+2)}}_{L^\infty} \right)^{s_{N+2}} \left\| u \right\|_{Q^{s_{N+2}}_{\partial_x}} \, .
\end{equation}
\end{proposition}

\begin{proof}[Proof of Proposition~\ref{lemma:Z0:estimate}]
For $\eta \leq \sfrac{\czero}{\sqrt{2}}$, we wish to demonstrate that
\begin{equation}
\label{eq:z0:estimate:1}
    \int_{\Omega \times (0,\sfrac{\czero^2}{2})} ( \exp(\eta^2 X_0)u - u )^2 \,dx\,dy\,dt + \int_{\Omega \times (\sfrac{\czero^2}{2},\czero^2)} ( \exp(-\eta^2 X_0)u - u )^2 \,dx\,dy\,dt \les \eta^2  M[u] ^2\,, 
\end{equation}
where we have used $M[u]$ to denote the right hand side of \eqref{eq:z0:estimate:1}. 
We focus on estimating the first term on the left-hand side of \eqref{eq:z0:estimate:1}, since the second term is similar. We split up the integrand as
\begin{align}
    &\int_{\Omega \times (0,\sfrac{\czero^2}{2})} \left(\exp(\eta^2 X_0)u - u\right)^2 \,dx\,dy\,dt \notag\\
    &\quad\lesssim \int_{\Omega \times (0,\sfrac{\czero^2}{2})} \left(\exp(\eta^2 X_0)u - \exp(\eta^2 X_0) \phi_\eta u\right)^2 \,dx\,dy\,dt \label{eq:firsttermofthisthing} \\
    & \quad\quad + \int_{\Omega \times (0,\sfrac{\czero^2}{2})} \left(\exp(\eta^2 X_0) \phi_\eta u -\phi_\eta u \right)^2 \,dx\,dy\,dt \label{eq:secondtermofthisthing}\\
    & \quad\quad + \int_{\Omega \times (0,\sfrac{\czero^2}{2})} \left( \phi_\eta u - u \right)^2 \,dx\,dy\,dt \, . \label{eq:thirdtermofthisthing}
\end{align}
The first and third terms on the right-hand side may be bounded by appealing to \eqref{eq:moll:2}. Focusing then on the second term and substituting in the variable $\tau$ for values $\tau\leq\eta^2$, it will suffice to show that 
\begin{equation}\notag
    \frac{d}{d\tau} \int_{\Omega \times (0,\sfrac{\czero^2}{2})} \left(\exp(\tau X_0)\phi_\eta u - \phi_\eta u\right)^2 \,dx\,dy\,dt \les  M[u]^2 \, .
\end{equation}
Calculating the left-hand side, we write that
\begin{align}
   &\frac{d}{d\tau} \int_{\Omega \times (0,\sfrac{\czero^2}{2})} \left(\exp(\tau X_0)\phi_\eta u- \phi_\eta u\right)^2 \,dx\,dy\,dt \notag\\
   &\quad = 2\int_{\Omega \times(0,\sfrac{\czero^2}{2})} \left[ \phi_\eta u(t+\tau,x+\tau b(y),y) - \phi_\eta u(t,x,y) \right] \notag\\
   &\qquad \qquad \times \left(\partial_t+b(y)\partial_x\right)\phi_\eta u(t+\tau,x+\tau b(y), y)  \,dx\,dy\,dt \notag\\
   &\quad = 2\int_{\Omega \times(0,\sfrac{\czero^2}{2})} \left[\phi_\eta u(t,x,y)-\phi_\eta u(t-\tau,x-\tau b(y),y) \right]  \left(\partial_t+b(y)\partial_x\right)\phi_\eta u(t,x, y) \,dx\,dy\,dt \notag\\
   &\quad = 2\int_{\Omega \times (0,\sfrac{\czero^2}{2})} \left[\phi_\eta u(t,x,y)-\phi_\eta u(t-\tau,x-\tau b(y),y) \right] \phi_\eta \left( X_0 u\right)(t,x, y)  \,dx\,dy\,dt \notag\\
   &\qquad + 2\int_{\Omega \times(0,\sfrac{\czero^2}{2})} \left[\phi_\eta u(t,x,y)-\phi_\eta u(t-\tau,x-\tau b(y),y) \right] \left[X_0, \phi_\eta\right]u(t,x, y) \,dx\,dy\,dt\, . \notag 
\end{align}
The second term in the last equality vanishes after appealing to \eqref{eq:moll:1}. To estimate the first term, first note that from \eqref{eq:phi:k}, we may convert the convolution of $\left(\partial_t+b\partial_x\right)u$ with $\phi_\eta$ into convolution with $\phi_\eta$ on the other term in the integrand. We then appeal to \eqref{eq:moll:1} to commute $\phi_\eta$ past the $\exp(-\tau X_0)$. Using the bound on $\partial_t u + b\partial_x u$ in $\Y^*$, we see that the proof will be complete if we can estimate
$$ Y \left( \phi_\eta\phi_\eta u - \exp\left(-\tau X_0\right) \phi_\eta \phi_\eta u \right)\, . $$
Appealing to Lemma~\ref{taylorseriesidentities} to rewrite this expression and then to \eqref{eq:moll:3} concludes the proof. 
\end{proof}

Using  Proposition~\ref{lemma:Z0:estimate}, we may immediately obtain an improved subelliptic estimate. Indeed, after employing the higher bracket inequality to absorb the $\norm{u}_{\bZ_k}$ terms as in the proof of Corollary~\ref{cor:improvedbrackets} and the subelliptic estimate to absorb the $\left\| u \right\|_{Q^{s_{N+2}}_{\partial_x}}$ term as in the proof of~\eqref{eq:subellipticestimate}, we obtain the following.

\begin{theorem}[Improved subelliptic estimate]
\label{thm:improvedsubellipticestimate}
For all $\nu \leq \nu_0 \ll_b 1$, we have that
\begin{equation}
\label{eq:improvedsubellipticestimate}
	\nu^{\frac{N+1}{2(N+3)}} \left( \norm{u-\langle u \rangle}_{L^2} + \norm{u}_{Q^{s_{N+1}}_{\p_x}} \right) \les_b  \left\| X_0 u\right\|_{\Y^*} + \norm{u}_{\Y}  \, 
	\end{equation}
for an implicit constant which is independent of $\nu$ but depends on $b$.
\end{theorem}

\section{Proofs of Theorems~\ref{thm:enhancement} and~\ref{thm:smoothing}}
\label{sec:proofofmaintheorem}

\begin{proof}[Proof of Theorem~\ref{thm:enhancement}]
For $t \geq 0$, we  define the time-translated seminorm
\begin{equation}
    \notag
    \| u\|_{\Y(t)} :=  \nu^{\sfrac{1}{2}} \| \p_y u \|_{L^2(\Omega \times (t, t + c_0^2))} \, ,
\end{equation}
with an analogous definition for $\Y^{*}(t)$.
For solutions of~\eqref{eq:fpde}, we have the identity 
\begin{equation}
    \label{eq:energy:est:fpde}
    \| f (\cdot,t_2)\|_{L^2(\Omega)}^2  - \| f(\cdot,t_1) \|_{L^2(\Omega)}^2 = - 2  \int_{t_1}^{t_2} \| Y f(\cdot,s) \|_{L^2(\Omega)}^2 ds \, .
\end{equation}
 For $k \in \mathbb{N}$, setting $t_1 = 0$ and $t_2 = kc_0^2 $ we can deduce from \eqref{eq:energy:est:fpde} that  
\begin{equation}\notag
	\sum_{i = 0}^{k-1}\norm{f}_{\Y(ic_0^2)}^2 \leq  \norm{f_{\rm in}}_{L^2(\Omega)}^2 \, .
\end{equation}
We now have the crucial estimate
\begin{equation}
\begin{aligned}
\label{eq:X0:bound}
\|X_0 f \|_{ \Y^* } \overset{\eqref{eq:fpde}}{=} \| Y^2 f \|_{\Y^*} &= \sup_u \left|\int_{\Omega \times (0,c_0^2)} Y^2f \cdot u  \, dx \, dy \, dt \right| \\
&\overset{\footnotesize\rm IBP}{=} \sup_u \left|\int_{\Omega \times (0,c_0^2)} Yf \cdot  Yu \, dx \, dy dt \right| \leq  \|Y f \|_{L^2(\Omega \times (0,c_0^2)}
\end{aligned}
\end{equation}
where the supremum is over $u \in L^2_{t,x} H^1_y(\Omega \times (0,c_0^2))$ ($(H^1_0)_y$ in the Dirichlet case) satisfying $\| u \|_{\Y} = 1$, see Definition~\ref{def:besov:spaces}. Notice that the boundary conditions satisfied by $f$ and the test function $u$ allow the above integration by parts.
Applying~\eqref{eq:X0:bound} and Theorem~\ref{thm:improvedsubellipticestimate} (Improved subelliptic estimate) to time translates of the solution, we obtain 
\begin{align}
\nu^{\frac{N+1}{N + 3}} \| f \|_{L^2(\Omega \times (0, kc_0^2))}^2 &=   \nu^{\frac{N+1}{N + 3}}\sum_{i = 0}^{k-1}  \| f \|_{L^2(\Omega \times ( ic_0^2,  (i + 1) c_0^2))}^2 \notag\\
&\leq 2 C_b \sum_{i = 0}^{k-1}\norm{f}_{\Y(ic_0^2)}^2 \notag\\
&\leq C_b\norm{f_{\rm in}}_{L^2(\Omega)}^2 \, . \label{eq:immacombine1}
\end{align}
We estimate the left-hand side from below using the monotonicity of the $L^2$ norm for solutions of~\eqref{eq:fpde} and the choice $\czero^2=\nu^{-\frac{N+1}{N+3}}$ from Corollary~\ref{cor:improvedbrackets}:
\begin{equation}\label{eq:estimatefrombelow}
    k \| f( \cdot,  kc_0^2) \|_{L^2(\Omega)}^2 \leq k c_0^2 \nu^{\frac{N+1}{N + 3}} \inf_{t' \in (0,kc_0^2)} \| f(\cdot,t') \|_{L^2(\Omega)}^2 \leq \nu^{\frac{N+1}{N + 3}} \| f \|_{L^2(\Omega \times (0, kc_0^2))}^2 \, .
\end{equation}
Combining~\eqref{eq:estimatefrombelow} with~\eqref{eq:immacombine1}, we have
\begin{equation}\notag
k \| f( \cdot, kc_0^2) \|_{L^2(\Omega)}^2 \leq C_b  \| f_{\rm in} \|_{L^2(\Omega)}^2 \, .
\end{equation}
Taking $k_0$ sufficiently large yields
\begin{equation}\notag
\| f( \cdot,  k_0c_0^2)\|_{L^2(\Omega)}^2 \leq \frac{1}{e} \norm{f_{\rm in}}_{L^2(\Omega)}^2 \, .
\end{equation}
We iterate this estimate to obtain, for an integer multiple $j \geq 0$ of $k_0c_0^2$,
\begin{equation}
    \notag
    \| f ( \cdot , j k_0 c_0^2) \|_{L^2(\Omega)}^2 \leq e^{-j} \| f_{\rm in} \|_{L^2(\Omega)}^2 \, .
\end{equation}
The monotonicity of the $L^2$ norm for solutions of~\eqref{eq:fpde} gives
\begin{equation}\notag
    \| f(\cdot,t) \|_{L^2(\Omega)}^2 \leq e \exp\left(-t k_0^{-1}\nu^{\frac{N+1}{N+3}}\right) \| f_{\rm in}\|_{L^2(\Omega)}^2 \, ,
\end{equation}
for any $t \geq 0$, which concludes the proof of the theorem.\end{proof}

\begin{proof}[Proof of Theorem~\ref{thm:smoothing}]. 
Let $q \: [0,+\infty) \to [0,+\infty)$ be a decreasing function satisfying
\begin{equation}\notag
	\sum_{j \in \N} q^2(2^j) < +\infty\, .
\end{equation}
For example, we may set $q(j) = ( \log (2 + |j|) )^{-\frac{1}{2} - \varepsilon}$.
We require the following characterization of the Besov space $Q^{s}_{\p_x}$.\footnote{For discussion of real-valued Besov spaces and equivalences, one may refer to Bahouri, Chemin, and Danchin \cite{bcd}. For vector-valued and Banach space-valued analogues, one may refer to Amann~\cite{amann97}, which contains \eqref{eq:besov:equivalence} in Section~5.} Let $(\dot \Delta_j^x)_{j \geq 0}$ be a suitable sequence of Littlewood-Paley projections onto dyadic annuli in the $x$-direction in frequency space satisfying $\sum_{j \geq 0} \dot \Delta_j u \equiv u$ for functions $u \: \Omega\times(0,c_0^2) \to \R$ with $\la u \ra(y,t) \equiv 0$. Then
\begin{equation}\label{eq:besov:equivalence}
    \sup_{j \geq 0} \, 2^{js} \norm{\dot \Delta_j^x u}_{L^2(\Omega\times(0,\czero^2))} \sim_s \norm{u}_{Q^s_{\p_x}(\Omega\times(0,\czero^2))} \, .
\end{equation}
Using this equivalence, we have that
\begin{equation}
\begin{aligned}
 \norm{|\p_x|^{s} q(|\p_x|) u}_{L^2(\Omega\times(0,\czero^2))}^2 &= \sum_{k \in 2\pi \Z} |k|^{2s} q^2(|k|) \iint_{(0,1)\times(0,\czero^2)}\left|\hat{u}(k,y,t)\right|^2 \,dy \,dt \\
    &\les_s \sum_{j \geq 0} 2^{2js} q^2(2^j) \norm{\dot \Delta_j^x u}_{L^2(\Omega\times(0,\czero^2))}^2 \\
    &\les_{s,q} \norm{u}_{Q^s_{\p_x}(\Omega\times(0,\czero^2))}^2 \, . \notag
    \end{aligned}
\end{equation}

The subelliptic estimate \eqref{eq:subellipticestimate} combined with the energy estimates gives
\begin{equation}\notag
	\nu^{\frac{N+1}{N+3}}  \norm{f}_{Q^{s_{N+1}}_{\p_x}(\Omega\times(0,\czero^2))}^2 \les_b \norm{f_{\rm in}}_{L^2}^2 - \norm{f(\cdot,\czero^2)}_{L^2(\Omega)}^2 \, .
\end{equation}
By the energy estimate, 
 the above embeddings, and a sum over the intervals $(n\czero^2,(n+1)\czero^2)$ for $n\in\N_0$, we have
\begin{equation}\notag
	 \nu^{\frac{N+1}{N+3}} \norm{|\p_x|^{\frac{1}{N+3}} q(|\p_x|) f }_{L^2(\Omega \times \R^+)}^2 \les_b \norm{f_{\rm in}}_{L^2}^2 \, .
\end{equation}
Let $t > 0$. By the pigeonhole principle or Chebyshev's inequality, by choosing $C$ sufficiently large, we can ensure that there exists a measurable set $A \subset (0,t)$ with $|A| \geq t/2$ which consists of times $\bar{t} \in A \subset (0,t)$ satisfying
\begin{equation}\notag
	\norm{|\p_x|^{\frac{1}{N+3}} q(|\p_x|) f(\cdot,\bar{t})}_{L^2(\Omega)}^2 \leq C t^{-1} \nu^{-\frac{N+1}{N+3}} \norm{f_{\rm in}}_{L^2}^2 \, .
\end{equation}
Since $|\p_x|^{\frac{1}{N+3}} q(|\p_x|) f$ also satisfies \eqref{eq:fpde}, we have that the left-hand side is monotone decreasing in time. This guarantees the smoothing estimate
\begin{equation}\notag
	\norm{|\p_x|^{\frac{1}{N+3}} q(|\p_x|) f(\cdot,t)}_{L^2(\Omega)}^2 \leq C t^{-1} \nu^{-\frac{N+1}{N+3}}  \norm{f_{\rm in}}_{L^2}^2
\end{equation} 
for any $t > 0$. One may continue smoothing by applying the estimate again to the function $|\p_x|^{\frac{1}{N+3}} q(|\p_x|) f(\cdot,t)$. Specifically, we may obtain
\begin{equation}\notag
	\norm{|\p_x|^{\frac{2}{N+3}} q^2(|\p_x|) f(\cdot,t)}_{L^2(\Omega)} \leq C_0 2^{-1} t^{-1} \nu^{-\frac{N+1}{N+3}}  \norm{f_{\rm in}}_{L^2} \, ,
\end{equation}
which is more convenient for summing. For given $t>0$, we may iterate this smoothing estimate $k$ times on intervals of length $t/k$. This gives
\begin{equation}\notag
	C_0^{-k} t^{k} k^{-k} \nu^{\frac{N+1}{N+3} k} \norm{\left[ |\p_x|^{\frac{2}{N+3}} q^2(|\p_x|) \right]^k f(\cdot,t)}_{L^2(\Omega)} \leq 2^{-k} \norm{f_{\rm in}}_{L^2} \, .
\end{equation}
for all $t > 0$ and $k \geq 0$. By Stirling's formula for $k^{-k}$, the Taylor series of $\exp$, and the triangle inequality, we have
\begin{equation}\notag
	\norm{ \exp \left( \mu_0 \nu^{\frac{N+1}{N+3}} |\p_x|^{\frac{2}{N+3}} q^2(|\p_x|) t \right) f(\cdot,t) }_{L^2(\Omega)} \les \norm{f_{\rm in}}_{L^2} \, ,
\end{equation}
where $\mu_0 \ll C_0^{-1}$. This gives the desired Gevrey smoothing.
\end{proof}

\begin{appendix}

\numberwithin{theorem}{section}

\section{Proof of sample interpolation inequality}
\label{app:interpolationproof}

\begin{proof}
It suffices to prove the (apparently) weaker inequality
\begin{equation}
    \label{eq:weakerinequality}
    \| f \|_{Q^{\sfrac{1}{2}}_{y\p_x}} \les \| f \|_{L^2_x \dot H^1_y} + \| y \p_x f \|_{L^2_x \dot H^{-1}_y}
\end{equation}
for all $f \in C^\infty_0$. After rescaling in $y$, we have
\begin{equation}\notag
    \| f \|_{Q^{\sfrac{1}{2}}_{y \p_x}} \les \lambda \| f \|_{L^2_x \dot H^1_y} + \lambda^{-1} \| y \p_x f \|_{L^2_x \dot H^{-1}_y}.
\end{equation}
Optimizing $\lambda$ gives the full inequality. The optimization step is possible because, by assumption, each term on the RHS of~\eqref{eq:weakerinequality} is non-zero.

For $\ell > 0$, consider the mollification-in-$x$ at scale $\ell^3$:
\begin{equation}\notag
    S_\ell f(x,y) := \ell^{-3}  \int_{\R} f(x-s,y)  \varphi(\ell^{-3} s) \, ds.
\end{equation}
We will require the properties
\begin{equation}
    \label{eq:mollificationproperty1}
    \| S_\ell f - f \|_{L^2} \les \ell \| f \|_{Q^{\sfrac{1}{3}}_{\p_x}}
\end{equation}
and
\begin{equation}
    \label{eq:mollificationproperty2}
    \| \p_x S_\ell f \|_{L^2} \les \ell^{-2} \| f \|_{Q^{\sfrac{1}{3}}_{\p_x}}.
\end{equation}
Let $h \in \R$. We expand
\begin{equation}\notag
\begin{aligned}
 \| e^{h^2 y\p_x} f  - f\|_{L^2} &\leq \| e^{h^2 y\p_x} f - e^{h^2 y\p_x} S_\ell f \|_{L^2} \\
 &+ \| e^{h^2 y\p_x} S_\ell f -  S_\ell f \|_{L^2} \\
 &+ \| S_\ell f - f \|_{L^2} \, .
\end{aligned}
\end{equation}
The first and third terms on the right-hand side are estimated in the same way, namely, by the property~\eqref{eq:mollificationproperty1} of the mollification operator.
We focus on the second term. We have
\begin{equation}\notag
    \frac{d}{d h} \| e^{h^2 y\p_x} S_\ell f -  S_\ell f \|_{L^2}^2 = 4 h \langle e^{h^2 y \p_x} S_h f, e^{h^2 y \p_x} y \p_x S_h f \rangle = 4 h \langle e^{h^2 y \p_x} S_h^* S_h f, y\p_x f \rangle \, ,
\end{equation}
which is estimated by $\dot H^1/\dot H^{-1}$ duality:
\begin{equation}\notag
    h |\langle e^{h^2 y \p_x} S_h^* S_h f, y\p_x f \rangle| \les h \| e^{h^2 y \p_x} S_h^* S_h f \|_{L^2_x \dot H^1_y} \|  y\p_x f \|_{L^2_x \dot H^{-1}_y} \, .
\end{equation}
It remains to estimate $h \| \p_y e^{h^2 y \p_x} S_h^* S_h f \|_{L^2}$. We have
\begin{equation}\notag
    h \p_y e^{h^2 y \p_x} S_h^* S_h f = h^3 \p_x e^{h^2 y \p_x} S_h^* S_h f + h e^{h^2 y \p_x} \p_y S^*_h S_h f - h \p_y S^*_h S_h f \, .
\end{equation}
The second and third terms are estimated in $L^2$ by $h \| \p_y f \|_{L^2}$. For the 1st term, we allow $\p_x$ to hit the mollifier and use the property~\eqref{eq:mollificationproperty2}:
\begin{equation}\notag
    h^3 \| \p_x e^{h^2 y \p_x} S_h^* S_h f \|_{L^2} \les h^3 \ell^{-2} \| f \|_{Q^{\sfrac{1}{3}}_{\p_x}} \les \ell \| f \|_{Q^{\sfrac{1}{3}}_{\p_x}}
\end{equation}
for $h \leq \ell$.
This gives
\begin{equation}\notag
    \left| \frac{d}{d h} \| e^{h^2 y\p_x} S_\ell f -  S_\ell f \|_{L^2}^2 \right| \les (h \| f \|_{L^2_x \dot H^1_y} + \ell \| f \|_{Q^{\sfrac{1}{3}}_{\p_x}}) \|  y\p_x f \|_{L^2_x \dot H^{-1}_y} \, .
\end{equation}
In particular, for $h = \ell$, we have
\begin{equation}\notag
    \| e^{h^2 y\p_x} S_\ell f -  S_\ell f \|_{L^2}^2 \les h^2 (\| f \|_{L^2_x \dot H^1_y} + \| f \|_{Q^{\sfrac{1}{3}}_{\p_x}}) \|  y\p_x f \|_{L^2_x \dot H^{-1}_y} \, .
\end{equation}
Combining all estimates with $h = \ell$, we have
\begin{equation}\notag
    \| e^{h^2 y\p_x} f  - f\|_{L^2}^2 \les h^2 \| f \|_{Q^{\sfrac{1}{3}}_{\p_x}}^2 + h^2 (\| f \|_{L^2_x \dot H^1_y} + \| f \|_{Q^{\sfrac{1}{3}}_{\p_x}}) \|  y\p_x f \|_{L^2_x \dot H^{-1}_y}
\end{equation}
or
\begin{equation}\notag
\begin{aligned}
    \| f \|_{Q^{\sfrac{1}{2}}_{y \p_x}}^2 &\les \| f \|_{Q^{\sfrac{1}{3}}_{\p_x}}^2 +  (\| f \|_{L^2_x \dot H^1_y} + \| f \|_{Q^{\sfrac{1}{3}}_{\p_x}}) \|  y\p_x f \|_{L^2_x \dot H^{-1}_y} \\
    &\les \| f \|_{Q^1_{\p_y}}^{\sfrac{2}{3}} \| f \|_{Q^{\sfrac{1}{2}}_{y \p_x}}^{\sfrac{4}{3}} + \| f \|_{L^2_x \dot H^1_y} \| y \p_x f \|_{L^2_x \dot H^{-1}_y} + \| f \|_{Q^1_{\p_y}}^{\sfrac{1}{3}} \| f \|_{Q^{\sfrac{1}{2}}_{y \p_x}}^{\sfrac{2}{3}} \| y \p_x f \|_{L^2_x \dot H^{-1}_y} \, ,
    \end{aligned}
\end{equation}
Then Young's inequality produces~\eqref{eq:weakerinequality}.
\end{proof}
\end{appendix}

\bibliographystyle{abbrv}
\bibliography{dissipation}

\begin{thebibliography}{10}

\bibitem{amann97}
H.~Amann.
\newblock Operator-valued {F}ourier multipliers, vector-valued {B}esov spaces,
  and applications.
\newblock {\em Math. Nachr.}, 186:5--56, 1997.

\bibitem{aris1956dispersion}
R.~Aris.
\newblock On the dispersion of a solute in a fluid flowing through a tube.
\newblock {\em Proceedings of the Royal Society of London. Series A.
  Mathematical and Physical Sciences}, 235(1200):67--77, 1956.

\bibitem{am19}
S.~Armstrong and J.-C. Mourrat.
\newblock Variational methods for the kinetic fokker-planck equation, 2019.

\bibitem{bcd}
H.~Bahouri, J.-Y. Chemin, and R.~Danchin.
\newblock {\em Fourier analysis and nonlinear partial differential equations},
  volume 343 of {\em Grundlehren der Mathematischen Wissenschaften [Fundamental
  Principles of Mathematical Sciences]}.
\newblock Springer, Heidelberg, 2011.

\bibitem{beckchaudharywayne}
M.~Beck, O.~Chaudhary, and C.~E. Wayne.
\newblock Rigorous justification of {T}aylor dispersion via center manifolds
  and hypocoercivity.
\newblock {\em Arch. Ration. Mech. Anal.}, 235(2):1105--1149, 2020.

\bibitem{beckwaynebar}
M.~Beck and C.~E. Wayne.
\newblock Metastability and rapid convergence to quasi-stationary bar states
  for the two-dimensional {N}avier-{S}tokes equations.
\newblock {\em Proc. Roy. Soc. Edinburgh Sect. A}, 143(5):905--927, 2013.

\bibitem{bedrossian2020regularity}
J.~Bedrossian, A.~Blumenthal, and S.~Punshon-Smith.
\newblock A regularity method for lower bounds on the {L}yapunov exponent for
  stochastic differential equations.
\newblock {\em arXiv preprint arXiv:2007.15827}, 2020.

\bibitem{jacobmicheleshear}
J.~Bedrossian and M.~Coti~Zelati.
\newblock Enhanced dissipation, hypoellipticity, and anomalous small noise
  inviscid limits in shear flows.
\newblock {\em Arch. Ration. Mech. Anal.}, 224(3):1161--1204, 2017.

\bibitem{bedrossian2020quantitative}
J.~Bedrossian and K.~Liss.
\newblock Quantitative spectral gaps and uniform lower bounds in the small
  noise limit for {M}arkov semigroups generated by hypoelliptic stochastic
  differential equations.
\newblock {\em arXiv preprint arXiv:2007.13297}, 2020.

\bibitem{bedrossianmasmoudiinviscid}
J.~Bedrossian and N.~Masmoudi.
\newblock Inviscid damping and the asymptotic stability of planar shear flows
  in the 2{D} {E}uler equations.
\newblock {\em Publ. Math. Inst. Hautes \'{E}tudes Sci.}, 122:195--300, 2015.

\bibitem{bmv1}
J.~Bedrossian, N.~Masmoudi, and V.~Vicol.
\newblock Enhanced dissipation and inviscid damping in the inviscid limit of
  the {N}avier-{S}tokes equations near the two dimensional {C}ouette flow.
\newblock {\em Arch. Ration. Mech. Anal.}, 219(3):1087--1159, 2016.

\bibitem{bedrossianvicolwang}
J.~Bedrossian, V.~Vicol, and F.~Wang.
\newblock The {S}obolev stability threshold for 2{D} shear flows near
  {C}ouette.
\newblock {\em J. Nonlinear Sci.}, 28(6):2051--2075, 2018.

\bibitem{bramanti}
M.~Bramanti.
\newblock {\em An invitation to hypoelliptic operators and {H}\"{o}rmander's
  vector fields}.
\newblock SpringerBriefs in Mathematics. Springer, Cham, 2014.

\bibitem{constantinkiselevryzhikzlatos}
P.~Constantin, A.~Kiselev, L.~Ryzhik, and A.~Zlato\v{s}.
\newblock Diffusion and mixing in fluid flow.
\newblock {\em Ann. of Math. (2)}, 168(2):643--674, 2008.

\bibitem{stablemixingestimates}
M.~Coti~Zelati.
\newblock Stable mixing estimates in the infinite {P}\'{e}clet number limit.
\newblock {\em J. Funct. Anal.}, 279(4):108562, 25, 2020.

\bibitem{ontherelationshipcpam}
M.~Coti~Zelati, M.~G. Delgadino, and T.~M. Elgindi.
\newblock On the relation between enhanced dissipation timescales and mixing
  rates.
\newblock {\em Comm. Pure Appl. Math.}, 73(6):1205--1244, 2020.

\bibitem{michelesquared}
M.~Coti~Zelati and M.~Dolce.
\newblock Separation of time-scales in drift-diffusion equations on
  {$\Bbb{R}^2$}.
\newblock {\em J. Math. Pures Appl. (9)}, 142:58--75, 2020.

\bibitem{drivascotizelati}
M.~Coti~Zelati and T.~D. Drivas.
\newblock A stochastic approach to enhanced diffusion.
\newblock {\em arXiv preprint arXiv:1911.09995}, 2019.

\bibitem{michelepoiseuille}
M.~Coti~Zelati, T.~M. Elgindi, and K.~Widmayer.
\newblock Enhanced dissipation in the {N}avier-{S}tokes equations near the
  {P}oiseuille flow.
\newblock {\em Comm. Math. Phys.}, 378(2):987--1010, 2020.

\bibitem{wendeng2}
W.~Deng.
\newblock Pseudospectrum for {O}seen vortices operators.
\newblock {\em Int. Math. Res. Not. IMRN}, 9:1935--1999, 2013.

\bibitem{wendeng1}
W.~Deng.
\newblock Resolvent estimates for a two-dimensional non-self-adjoint operator.
\newblock {\em Commun. Pure Appl. Anal.}, 12(1):547--596, 2013.

\bibitem{evans}
L.~C. Evans.
\newblock {\em Partial differential equations}, volume~19 of {\em Graduate
  Studies in Mathematics}.
\newblock American Mathematical Society, Providence, RI, second edition, 2010.

\bibitem{lunardi2006}
B.~Farkas and A.~Lunardi.
\newblock Maximal regularity for {K}olmogorov operators in {$L^2$} spaces with
  respect to invariant measures.
\newblock {\em J. Math. Pures Appl. (9)}, 86(4):310--321, 2006.

\bibitem{gallaghergallaynier}
I.~Gallagher, T.~Gallay, and F.~Nier.
\newblock Spectral asymptotics for large skew-symmetric perturbations of the
  harmonic oscillator.
\newblock {\em Int. Math. Res. Not. IMRN}, (12):2147--2199, 2009.

\bibitem{gallayaxisymmetrization}
T.~Gallay.
\newblock Enhanced dissipation and axisymmetrization of two-dimensional viscous
  vortices.
\newblock {\em Arch. Ration. Mech. Anal.}, 230(3):939--975, 2018.

\bibitem{gimv}
F.~Golse, C.~Imbert, C.~Mouhot, and A.~F. Vasseur.
\newblock Harnack inequality for kinetic {F}okker-{P}lanck equations with rough
  coefficients and application to the {L}andau equation.
\newblock {\em Ann. Sc. Norm. Super. Pisa Cl. Sci. (5)}, 19(1):253--295, 2019.

\bibitem{guolandau}
Y.~Guo.
\newblock The {L}andau equation in a periodic box.
\newblock {\em Comm. Math. Phys.}, 231(3):391--434, 2002.

\bibitem{Hairermalliavin}
M.~Hairer.
\newblock On {M}alliavin's proof of {H}\"{o}rmander's theorem.
\newblock {\em Bull. Sci. Math.}, 135(6-7):650--666, 2011.

\bibitem{siminghefractional}
S.~He.
\newblock Enhanced dissipation, hypoellipticity for passive scalar equations
  with fractional dissipation.
\newblock {\em arXiv preprint arXiv:2103.07906}, 2021.

\bibitem{helffernier}
B.~Helffer and F.~Nier.
\newblock {\em Hypoelliptic estimates and spectral theory for {F}okker-{P}lanck
  operators and {W}itten {L}aplacians}, volume 1862 of {\em Lecture Notes in
  Mathematics}.
\newblock Springer-Verlag, Berlin, 2005.

\bibitem{hormander67}
L.~H\"{o}rmander.
\newblock Hypoelliptic second order differential equations.
\newblock {\em Acta Math.}, 119:147--171, 1967.

\bibitem{hormandervol3}
L.~H\"{o}rmander.
\newblock {\em The analysis of linear partial differential operators. {III}}.
\newblock Classics in Mathematics. Springer, Berlin, 2007.
\newblock Pseudo-differential operators, Reprint of the 1994 edition.

\bibitem{imbertsilvestre1}
C.~Imbert and L.~Silvestre.
\newblock The weak {H}arnack inequality for the {B}oltzmann equation without
  cut-off.
\newblock {\em J. Eur. Math. Soc. (JEMS)}, 22(2):507--592, 2020.

\bibitem{ionescujia}
A.~D. Ionescu and H.~Jia.
\newblock Inviscid damping near the {C}ouette flow in a channel.
\newblock {\em Comm. Math. Phys.}, 374(3):2015--2096, 2020.

\bibitem{kelvin1887stability}
L.~Kelvin.
\newblock Stability of fluid motion: rectilinear motion of viscous fluid
  between two parallel plates.
\newblock {\em Phil. Mag}, 24(5):188--196, 1887.

\bibitem{Kohnsproof}
J.~J. Kohn.
\newblock Pseudo-differential operators and hypoellipticity.
\newblock In {\em Partial differential equations ({P}roc. {S}ympos. {P}ure
  {M}ath., {V}ol. {XXIII}, {U}niv. {C}alifornia, {B}erkeley, {C}alif., 1971)},
  pages 61--69, 1973.

\bibitem{kolmogorov}
A.~Kolmogoroff.
\newblock Zuf\"{a}llige {B}ewegungen (zur {T}heorie der {B}rownschen
  {B}ewegung).
\newblock {\em Ann. of Math. (2)}, 35(1):116--117, 1934.

\bibitem{lernerbook}
N.~Lerner.
\newblock {\em Metrics on the phase space and non-selfadjoint
  pseudo-differential operators}, volume~3 of {\em Pseudo-Differential
  Operators. Theory and Applications}.
\newblock Birkh\"{a}user Verlag, Basel, 2010.

\bibitem{weioseen}
T.~Li, D.~Wei, and Z.~Zhang.
\newblock Pseudospectral and spectral bounds for the {O}seen vortices operator.
\newblock {\em Ann. Sci. \'{E}c. Norm. Sup\'{e}r. (4)}, 53(4):993--1035, 2020.

\bibitem{mouhot1}
C.~Mouhot.
\newblock De {G}iorgi--{N}ash--{M}oser and {H}\"{o}rmander theories: new
  interplays.
\newblock In {\em Proceedings of the {I}nternational {C}ongress of
  {M}athematicians---{R}io de {J}aneiro 2018. {V}ol. {III}. {I}nvited
  lectures}, pages 2467--2493. World Sci. Publ., Hackensack, NJ, 2018.

\bibitem{silvestre1}
L.~Silvestre.
\newblock A new regularization mechanism for the {B}oltzmann equation without
  cut-off.
\newblock {\em Comm. Math. Phys.}, 348(1):69--100, 2016.

\bibitem{taylor1954dispersion}
G.~I. Taylor.
\newblock The dispersion of matter in turbulent flow through a pipe.
\newblock {\em Proceedings of the Royal Society of London. Series A.
  Mathematical and Physical Sciences}, 223(1155):446--468, 1954.

\bibitem{villanihypocoercivity}
C.~Villani.
\newblock Hypocoercivity.
\newblock {\em Mem. Amer. Math. Soc.}, 202(950):iv+141, 2009.

\bibitem{dongyiweiresolvent}
D.~Wei.
\newblock Diffusion and mixing in fluid flow via the resolvent estimate.
\newblock {\em Sci. China Math.}, 64(3):507--518, 2021.

\bibitem{kolmogorovhypocoercivitywei}
D.~Wei and Z.~Zhang.
\newblock Enhanced dissipation for the {K}olmogorov flow via the hypocoercivity
  method.
\newblock {\em Sci. China Math.}, 62(6):1219--1232, 2019.

\bibitem{kolmogorovwei}
D.~Wei, Z.~Zhang, and W.~Zhao.
\newblock Linear inviscid damping and enhanced dissipation for the {K}olmogorov
  flow.
\newblock {\em Adv. Math.}, 362:106963, 103, 2020.

\end{thebibliography}

\end{document}